\documentclass{amsart}

\usepackage{graphicx,amsthm, amsmath, amssymb, mathrsfs}
\usepackage{mathtools}
\usepackage[us,12hr]{datetime}

\usepackage{tikz}

\usepackage[pagebackref,hypertexnames=false, colorlinks, citecolor=blue, linkcolor=blue]{hyperref}
\usepackage{enumerate}
\usepackage{paralist}

\newtheorem{theorem}{Theorem}[section]

\newtheorem{lemma}[theorem]{Lemma}

\theoremstyle{definition}

\newtheorem{cor}[theorem]{Corollary}
\newtheorem{remark}[theorem]{Remark}

\numberwithin{equation}{section}

\def\Xint#1{\mathchoice
   {\XXint\displaystyle\textstyle{#1}}%
   {\XXint\textstyle\scriptstyle{#1}}%
   {\XXint\scriptstyle\scriptscriptstyle{#1}}%
   {\XXint\scriptscriptstyle\scriptscriptstyle{#1}}%
   \!\int}
\def\XXint#1#2#3{{\setbox0=\hbox{$#1{#2#3}{\int}$}
     \vcenter{\hbox{$#2#3$}}\kern-.5\wd0}}

\def\dashint{\Xint-}

\makeatletter
\def\moverlay{\mathpalette\mov@rlay}
\def\mov@rlay#1#2{\leavevmode\vtop{%
   \baselineskip\z@skip \lineskiplimit-\maxdimen
   \ialign{\hfil$\m@th#1##$\hfil\cr#2\crcr}}}
\newcommand{\charfusion}[3][\mathord]{
    #1{\ifx#1\mathop\vphantom{#2}\fi
        \mathpalette\mov@rlay{#2\cr#3}
      }
    \ifx#1\mathop\expandafter\displaylimits\fi}
\makeatother

\newcommand{\R}{\mathbb R}
\newcommand{\al}{\alpha}
\newcommand{\M}{\mathcal{M}}
\newcommand{\1}{\mathbf{1}}
\newcommand{\D}{\mathscr D}

\newcommand{\Z}{\mathbb Z}
\newcommand{\A}{\mathsf A}

\begin{document}

\title[Bilinear fractional integral operator: A union condition]{Weighted estimates for a bilinear fractional integral operator and its commutators:\\ A union condition}

\author{Cong Hoang}
\address{Cong Hoang, Department of Mathematics, Florida A\&M University}
\email{cong.hoang@famu.edu}
%\thanks{* Corresponding author.}

%\author{Kabe Moen}
%\address{Kabe Moen, Department of Mathematics, University of Alabama}
%\curraddr{Department of Mathematics and Statistics,
%\email{xyz@math.university.edu}
%\thanks{The second author was supported by NSF Grant \#1201504.}

\allowdisplaybreaks
%    General info
%\subjclass[2010]{Primary 42B20, 26A33}


%% This command is to set no indent to the entire document
\setlength\parindent{0pt}


\maketitle

\begin{abstract} The main theme of this paper is to give sufficient conditions for the weighted boundedness of the bilinear fractional integral operator $\mathsf{BI}_\al$. The proposed condition involves the union of multilinear Muckenhoupt-type conditions. We have achieved new results in an unknown case and remarkably improved other known results by utilizing the hidden convolution nature inside the operator. We also study the effects of the general product commutators on the main operator and the weighted estimates for a related maximal operator that norm-wise dominates the main operator.
%\\\\ Last updated on \today.

\hfill

\end{abstract}

%-------------------------------------------------
% NEW SECTION: INTRODUCTION
%-------------------------------------------------
\section{Introduction}
In the 1990's, the bilinear fractional integral operator
$$\mathsf{BI}_\al(f,g)(x)= \int_{\R^n}\frac{f(x-y)\,g(x+y)}{|y|^{n-\al}}\,dy, \qquad 0<\al<n$$
was introduced by Kenig, Stein \cite{KenStein1999} and Grafakos \cite{Grafakos1992} as an operator that has close relations to the bilinear Hilbert transform
$$\mathsf{BH}(f,g)(x)=p.v.\int_{\mathbb{R}}\frac{f(x-y)\thinspace g(x+y)}{y}\thinspace dy.$$
They proved that for any pair of exponents $p_1,p_2\in(1,\infty)$, the operator $\mathsf{BI}_\al$ would map the product space $L^{p_1}(\mathbb{R}^n)\times L^{p_2}(\mathbb{R}^n)$ into $L^q(\mathbb{R}^n)$ where $q$ is computed via the equation $\frac1q=\frac{1}{p_1}+\frac{1}{p_2}-\frac{\al}{n}$. Those spaces are known as unweighted spaces. In reality, everything has different impact and importance, so weighted spaces such as $L^p(w)$ were then naturally considered, where $w$ is an assigned weight for the space. Hereafter, by a weight we mean a non-negative and locally integrable function. We shall discuss the conditions for the mapping
$$
\mathsf{BI}_{\alpha}: L^{p_1}(v_1)\times L^{p_2}(v_2)\longrightarrow L^{q}(u)
$$
to hold true. Such mapping is also referred to as the weighted norm estimate or the weighted boundedness of $\mathsf{BI}_\al$. When $u^\frac{1}{q}=v_1^\frac{1}{p_1}v_2^\frac{1}{p_2}$ and $\frac1q=\frac{1}{p_1}+\frac{1}{p_2}-\frac{\al}{n}$, the estimate is said to be a 1-vector-weight estimate, otherwise it is called a 2-vector-weight estimate. 

Weighted estimates for $\mathsf{BI}_\al$ were pretty much unknown until 2014 when Moen \cite{Kabe2014} published some initial results for the case when $p\coloneq\frac{p_1p_2}{p_1+p_2}\leqslant q\leqslant 1$. In that paper, a dyadic version of $\mathsf{BI}_\al$, namely
$$\mathsf{BI}_\al^\D(f,g)(x)=\sum_{Q\in\D}|Q|^{\frac{\al}{n}-1}\int_{|y|_\infty\leqslant\ell(Q)}f(x-y)\,g(x+y)\,dy\,\1_Q(x),$$
was introduced and proved to be point-wise equivalent to $\mathsf{BI}_\al$ when $f,g\geqslant0$. Here, $\D$ is a dyadic grid in $\R^n$ whose precise definition is given in Section \ref{pre}. In 2017, Hoang and Moen \cite{KC2017} generalized these results and extended them to the case when $1<p\leqslant q$. In 2019, Komori-Furuya \cite{Komori2019} proved a necessary and sufficient condition for the weighted boundedness of $\mathsf{BI}_\al$ in 1-vector-weight settings, and the weights were limited to power weights of the form $|x|^\gamma$. For general weights, the full picture is still wide open. The study of $\mathsf{BI}_\al$ has also drawn the interest of other mathematicians: He and Yan \cite{HY2020}, Ghosh and Singh \cite{GS2023}, et al.

\begin{figure}[h]
    \begin{tikzpicture}
    \draw[black, thin,->] (-5mm,0mm) -- (55mm,0mm) node [black,thick,anchor=west] {$p$};
    \draw[black, thin,->] (0mm,-5mm) -- (0mm,55mm) node [black,thick,anchor=east] {$q$};
    \draw[black, thin,dashed] (10mm,55mm) -- (10mm,0mm) node [black,thick,anchor=north] {$\frac{1}{2}$};
    \draw[black, thin,dashed] (0mm,0mm) -- (55mm,55mm);
    \draw[black, thin,dashed] (20mm,55mm) -- (20mm,0mm) node [black,thick,anchor=north] {$1$};
    \draw[black, thin,dashed] (20mm,20mm) -- (0mm,20mm) node [black,thick,anchor=east] {$1$};
    \draw[black,thick, dashed] (10mm,10mm) -- (20mm,20mm) -- (10mm,20mm) -- (10mm,10mm);
    \filldraw[gray] (10mm,10mm) -- (20mm,20mm) -- (10mm,20mm) -- (10mm,10mm);
    \draw[black,thick, dashed] (55mm,55mm) -- (20mm,20mm) -- (20mm,55mm);
    \filldraw[gray] (20mm,20mm) -- (55mm,55mm) -- (20mm,55mm) -- (20mm,20mm);
    \draw[black,thick,dashed] (10mm,55mm) -- (10mm,20mm) -- (20mm,20mm) -- (20mm,55mm);
    \draw[black, thin,dashed] (15mm,37mm) -- (15mm,37mm) node [black,rotate=90] {$p<1<q$ ?};
    \filldraw[gray!50, opacity = 0.5] (10mm,55mm) -- (10mm,20mm) -- (20mm,20mm) -- (20mm,55mm);
    \end{tikzpicture}
    \caption{The missing case for general weights.}
\end{figure}

In this paper, we give a sufficient condition for the boundedness of $\mathsf{BI}_\al$ in 2-vector-weight settings when $p<1<q$, as stated in Theorem \ref{p1q}. We also introduce better conditions that allow many more possible weights in the other two cases $p\leqslant q\leqslant1$ and $1\leqslant p\leqslant q$, as in Theorems \ref{1pq} and \ref{pq1}. Besides, we also obtain a Maximal Control Theorem when $q\leqslant1$; see Theorem \ref{MaxControl}. The key that leads to all of these achievements lies in the hidden convolution nature of the operator. The idea is made precise as follows: for $f,g\geqslant0$, we have
\begin{align*}
    \mathsf{BI}_\al(f,g)(x) & \simeq \mathsf{BI}_\al^\D(f,g)(x) \\
    & = \sum_{Q\in\D}|Q|^{\frac{\al}{n}-1}\int_{\R^n}f(2x-z)\,g(z)\,\1_{\left[-\ell(Q),\ell(Q)\right]^n}(z-x)\,dz\,\1_Q(x).
\end{align*}
Let $3Q$ denote the cube whose center is the center of $Q$ and side-length is three times as long, then we have (see Figure \ref{ChangeVar} to visually understand the next estimates)
\begin{align*}
    \mathsf{BI}_\al(f,g)(x) & \lesssim \sum_{Q\in\D}|Q|^{\frac{\al}{n}-1}\int_{\R^n}f(2x-z)\,g(z)\,\1_{3Q}(2x-z)\,\1_{3Q}(z)\,dz\,\1_Q(x),
\end{align*}
and hence
\begin{align} \label{ConvolutionDomination}
    \mathsf{BI}_\al(f,g)(x) & \lesssim \sum_{Q\in\D}|Q|^{\frac{\al}{n}-1}\left[(f\1_{3Q})*(g\1_{3Q})\right](2x)\,\1_Q(x).
\end{align}

\begin{figure}[h]
    \begin{tikzpicture}[scale=0.7]
    \draw[black, thin,->] (-14mm,24mm) -- (38mm,24mm) node [black,thick,anchor=west] {$x$};
    \draw[black, thin,->] (-10mm,-4mm) -- (-10mm,60mm) node [black,thick,anchor=east] {$y$};
    \filldraw[gray, thick] (14mm,8mm) -- (30mm,8mm) -- (30mm,40mm) -- (14mm,40mm) -- (14mm,8mm);
    \draw[white, thick,<->] (14mm,24mm) -- (30mm,24mm);
    \draw[white, thick] (22mm,15mm) -- (22mm,15mm) node [white,thick,anchor=south] {$Q$};
    \draw[black, thin,dashed] (14mm,8mm) -- (-10mm,8mm) node [black,thick,anchor=east] {$-\ell(Q)$};
    \draw[black, thin,dashed] (14mm,40mm) -- (-10mm,40mm) node [black,thick,anchor=east] {$-\ell(Q)$};

    \draw[black, thick,->] (44mm,32mm) -- (72mm,32mm);
    \draw[black, thin] (58mm,38mm) -- (58mm,38mm) node [black] {$\substack{u\,=\,x-y\,=\,2x-z \\ v\,=\,x+y\,=\,z}$};

    \draw[black, thin,->] (76mm,0mm) -- (140mm,0mm) node [black,thick,anchor=west] {$u$};
    \draw[black, thin,->] (80mm,-4mm) -- (80mm,60mm) node [black,thick,anchor=east] {$v$};
    \filldraw[gray, thick] (88mm,40mm) -- (120mm,8mm) -- (136mm,24mm) -- (104mm,56mm) -- (88mm,40mm);
    \draw[black,thin] (136mm,56mm) -- (136mm,8mm) -- (88mm,8mm) -- (88mm,56mm) -- (136mm,56mm) node [black,thick,anchor=south] {$3Q\times3Q$};
%    \draw[black, thin,->] (103mm,32mm) -- (103mm,32mm) node [white] {\tiny{$\substack{\frac{u+v}{2}\,=\,x\,\in\,Q \\ \quad\left|\frac{v-u}{2}\right|=|y|_\infty<\ell(Q)}$}};
    \end{tikzpicture}
    \caption{A visualization for: $\1_{\left[-\ell(Q),\ell(Q)\right]^n}(z-x)\leqslant\1_{3Q}(2x-z)\times\1_{3Q}(z)$.}\label{ChangeVar}
\end{figure}

\hfill

As we develop the new conditions for the boundedness of $\mathsf{BI}_\al$, we realize that the techniques we use can be readily adapted to handle the effects of the commutators on $\mathsf{BI}_\al$. This is interesting because the commutators would often increase the singularity of the operators. Given a linear operator $T$ and a function $b$, the commutator $[b,T]$ is defined to be
$$
[b,T]f=b \thinspace T(f)-T(bf).
$$
The commutators were introduced by Coifman, Rochberg and Weiss \cite{KRW1976} while studying the classical factorization theory of $H^p$ spaces. Commutators for bilinear operators are a bit more complicated to define, so we shall deter their definitions until the later sections when they will be investigated. Commutators of both linear and bilinear fractional integral operators are interesting topics for many mathematicians since then: Segovia and Torrea \cite{ST1991} Duong and Yan \cite{DY2004}, Chen and Wu \cite{ChW2013}, Cao and Xue \cite{CX2019}, Lu and Tao \cite{LT2023}, et al.

\hfill

The rest of the paper goes as follows:
\begin{itemize}
    \item Section \ref{pre} is dedicated to providing sufficient background on various tools and known knowledge that we shall need for our work.
    \item Section \ref{main} presents our main results for the weighted boundedness of $\mathsf{BI}_\al$ with detailed discussions on how they improve previous known results. We also discuss an immediate application of our results at the end of the section.
    \item Section \ref{ProofMain} shows the proof of our main results.
    \item Section \ref{Comm} investigates the effects of the commutators on the bilinear fractional integral operators and their weighted estimates.
    \item Section \ref{Max} studies the weighted boundedness of a related maximal operator and discusses a Maximal Control Theorem.
\end{itemize}

\hfill

%-------------------------------------------------
% NEW SECTION: PRELIMINARIES
%-------------------------------------------------
\section{Preliminaries}\label{pre}

One of the most important and innovative ideas in Analysis is the theory of dyadic grids and cubes. A dyadic grid $\mathscr{D}$ is a countable collection of cubes that satisfies the following properties:

\begin{enumerate}[\hspace{5mm}i)]
    \item For any cube $Q$ in $\mathscr{D}$, its length $\ell(Q)=2^k$ for some $k\in\mathbb{Z}$.
    \item For each $k\in\mathbb{Z}$, the set $\{Q\in\mathscr{D}:\thinspace \ell(Q)=2^k\}$ forms a partition of $\mathbb{R}^n$.
    \item For any two cubes $Q,P$ in $\mathscr{D}$, we have $Q\cap P=\emptyset$ or $P$ or $Q$.
\end{enumerate}

A common technical issue encountered when using the dyadic grid $\mathscr{D}$ is that not every cube in $\R^n$ can be contained in a dyadic cube from $\mathscr{D}$. To overcome this, the shifted grids of $\mathscr{D}$ were introduced:
$$\mathscr{D}_t=\left\{2^{-k}\left([0,1)^n+m+(-1)^kt\right):\thinspace k\in\mathbb{Z},m\in \mathbb{Z}^n\right\}, \quad t\in\{0,1/3\}^n.$$
This idea is made precise by the following theorem in \cite{Lerner2013}.

\begin{theorem}[also known as the $\frac{1}{3}$-trick] \label{1/3trick}
Given any cube $Q$ in $\mathbb{R}^n$, there exists a $t\in \{0,1/3\}^n$ and a cube $P\in\mathscr{D}_t$ such that $Q\subset P$ and $\ell(P)\leqslant 6\thinspace \ell(Q)$.
\end{theorem}

For the purpose of simplicity, we shall re-index the $t$ in the above theorem as $t\in\{1,...,2^n\}$. Another important and very useful concept is: the sparse family of cubes. A family of cubes $\mathscr{S}$ is said to be sparse if for any cube $Q\in\mathscr{S}$, there exists a set $E_Q\subset Q$ such that the family $\{E_Q\}_{Q\in\mathscr{S}}$ is pairwise disjoint and $|Q|\leqslant2|E_Q|$.

\hfill

A function $\Phi:[0,\infty)\to[0,\infty)$ is called a Young function if it is convex, continuous, strictly increasing, $\Phi(0)=0$ and $\frac{\Phi(t)}{t}\rightarrow\infty$ as $t\rightarrow\infty$. For every Young function $\Phi$, there exists an associate Young function $\overline{\Phi}$ such that $\Phi^{-1}(t)\,\overline{\Phi}^{-1}(t)\approx t$. Interested readers may find more information about Young functions from \cite{UMP2011}. Given a Young function $\Phi$, the Orlicz average of $f$ over a cube $Q$ is defined as
$$\|f\|_{\Phi,Q}=\inf\left\{\lambda>0:\thinspace\dashint_{Q}\Phi\left(\frac{|f(x)|}{\lambda}\right)dx\leqslant1\right\}$$
where $\dashint_Q=\frac{1}{|Q|}\int_Q$. Krasnosel'ski\u{i} and Ruticki\u{i} \cite{KR1961} proved that $\|f\|_{\Phi,Q}$ is equivalent to
$$
\|f\|'_{\Phi,Q}=\inf_{\lambda>0}\left\{\lambda+\frac{\lambda}{|Q|}\int_Q\Phi\left(\frac{|f(x)|}{\lambda}\right)dx\right\}.
$$
More precisely, we have
$$
\|f\|_{\Phi,Q}\leqslant\|f\|'_{\Phi,Q}\leqslant2\|f\|_{\Phi,Q}.
$$
When $\Phi(t)=t^p$ with $p>1$, we have
$$\|f\|_{\Phi,Q}=\|f\|_{L^p,Q}=\left(\dashint_Q|f(x)|^pdx\right)^\frac{1}{p}.$$
It has also been a common practice to write $\|f\|_{\Phi,Q}=\|f\|_{\exp{L},Q}$ when  $\Phi(t)=e^t-1$, and $\|f\|_{\Phi,Q}=\|f\|_{L^r(\log{L})^s,Q}$ when  $\Phi(t)=t^r\log(1+t)^s$.

\hfill

Let $\textsf{\textsf{BMO}}$ be the collection of functions of bounded mean oscillation; i.e., functions $b$ that satisfies
$$
\|b\|_{\textsf{BMO}}\coloneq\sup_Q\dashint_Q|b(x)-b_Q| \thinspace dx < \infty
$$
where $b_Q=\dashint_Qb(x) \thinspace dx$. As a consequence of the John-Nirenberg theorem, $\textsf{BMO}$ functions satisfy the exponential integrability as stated in the following theorem.

\begin{theorem} \label{ExpoInt}
Given $b\in \mathsf{BMO}$, there exists a constant $c_n$ such that
$$
\sup_Q \dashint_Q\exp \left(\frac{|b(x)-b_Q|}{2^{n+2}\|b\|_{\mathsf{BMO}}}\right)dx \leqslant c_n
$$
for all cube $Q$. In particular, $\|b-b_Q\|_{\exp L,Q}\leqslant c_n 2^{n+2} \|b\|_{\mathsf{BMO}}$.
\end{theorem}
A proof of Theorem \ref{ExpoInt} can be found in \cite{Journe1983}.

\hfill

The Orlicz maximal function is defined to be
$$M_{\Phi}(f)(x)=\sup_{Q\ni x}\|f\|_{\Phi,Q}.$$

Given a Young function $\Phi$, we write $\Phi\in B_p$ if and only if there exists a real number $c>0$ such that
$$
\int_c^{\infty}\frac{\Phi(t)}{t^{p+1}}\thinspace dt < \infty.
$$

P\'erez \cite{Perez1995} gave a necessary and sufficient condition for the boundedness of these Orlicz maximal operators.

\begin{theorem}\label{Bp}
For any $p\in(1,\infty)$,
$$
\|M_{\Phi}f\|_{L^p(\mathbb{R}^n)}\leqslant C\thinspace \|f\|_{L^p(\mathbb{R}^n)}
$$
if and only if $\thinspace\Phi$ satisfies the $B_p$ condition.
\end{theorem}

Given a Young function $\Phi$, we write $\Phi\in B_{p,q}$ if and only if there exists a real number $c>0$ such that
$$
\int_c^{\infty}\frac{\Phi(t)^\frac{q}{p}}{t^{q+1}}\thinspace dt < \infty.
$$
It was shown in \cite{UK2013} that $B_p\subsetneq B_{p,q}$ for $1<p<q$.

\hfill

For each $0<\al<n$ and a Young function $\Phi$, the fractional Orlicz maximal function is defined by
$$M_{\al,\Phi}(f)(x)=\sup_{Q\ni x}|Q|^\frac{\al}{n}\|f\|_{\Phi,Q}.$$
Cruz-Uribe and Moen \cite{UK2013} proved the following theorem.

\begin{theorem}\label{Bpq}
Suppose $0<\al<n$. For any $p\in(1,\frac{n}{\al})$, let $q$ be such that $\frac{\al}{n}=\frac{1}{p}-\frac{1}{q}$, then we have
$$
\|M_{\al,\Phi}f\|_{L^q(\mathbb{R}^n)}\leqslant C\thinspace \|f\|_{L^p(\mathbb{R}^n)}
$$
if and only if $\thinspace\Phi$ satisfies the $B_{p,q}$ condition.
\end{theorem}

There is also a generalized H\"older inequality for these Orlicz averages.

\begin{lemma}
If $\Phi,\Psi,\Theta$ are Young functions such that
$$
\Phi^{-1}(t)\,\Psi^{-1}(t)\lesssim \Theta^{-1}(t), \hspace{2mm} \forall t\geqslant t_0\geqslant0
$$
then
$$
\|fg\|_{\Theta,Q}\lesssim \|f\|_{\Phi,Q} \, \|g\|_{\Psi,Q}.
$$
In particular, for any Young function $\Phi$,
$$
\dashint_{Q}{|f(x)\thinspace g(x)|\thinspace dx}\leqslant2\thinspace\|f\|_{\Phi,Q}\thinspace\|g\|_{\overline{\Phi},Q}.
$$
\end{lemma}

\hfill

The Muckenhoupt class of weights $\mathsf{A}_\infty=\cup_{p>1}\mathsf{A}_p$ where $\mathsf{A}_p$ is the collection of weights $w$ that satisfy
$$
\sup_Q\left(\dashint_Qw\right)\left(\dashint_Qw^{1-p'}\right)^{p-1}<\infty.
$$
When $p=1$, one says $w\in \mathsf{A}_1$ if $w$ satisfies $Mw(x)\leqslant C\thinspace w(x)$ for almost very $x\in \mathbb{R}^n$. From \cite{Duoan2000} we know the following facts.

\begin{lemma} \label{AinfinityProperty}
If $w\in \mathsf{A}_\infty$ then the followings hold:

i) for every $\eta\in(0,1)$, there exists $\kappa\in(0,1)$ such that: given a cube $Q$ and

\hfill $S\subseteq Q$ with $|S|\leqslant\eta\thinspace |Q|$, we will also have $w(S)\leqslant\kappa\thinspace w(Q)$;

ii)  there exist $\delta_0>1$ such that
$$
    \left(\dashint_Qw^{1+\delta}\right)^{\frac{1}{1+\delta}} \leqslant C \thinspace \dashint_Qw \qquad \text{for all} \quad 0<\delta\leqslant\delta_0.
$$
\end{lemma}

A bilinear version for $\mathsf{A}_p$ is the $\mathsf{A}_{[p_1,\,p_2],\,q}$ class for pairs of weights. A pair of weight $(w_1,w_2)$ is said to satisfy the $\mathsf{A}_{[p_1,\,p_2],\,q}$ condition if
$$\sup_Q \left(\dashint_Q w_1^qw_2^q\right)^{\frac{1}{q}}\left(\dashint_Q{w_1^{-p_1'}}\right)^{\frac{1}{p_1'}}\left(\dashint_Q{w_2^{-p_2'}}\right)^{\frac{1}{p_2'}} <\infty.$$
The following theorem was proved in \cite{CWX2014} and \cite{Kabe2009}.

\begin{theorem}\label{1wApqProperty}
If $1<p_1,p_2<\infty$, then $(w_1,w_2)\in \mathsf{A}_{[p_1,\,p_2],\,q}$ if and only if 
$$
(w_1w_2)^q\in \mathsf{A}_{2q} \qquad \text{and} \qquad w_i^{-p_i'}\in \mathsf{A}_{2p_i'}.
$$
\end{theorem}

Another important class of weights is the Reverse H\"older class. For $s>1$, a weight $w$ is said to be in the Reverse H\"older class of order $s$, denoted as $\mathsf{RH}_s$, if there exists a constant $C$ such that
$$
\left(\dashint_Qw^s\right)^{\frac{1}{s}} \leqslant C \thinspace \dashint_Qw \qquad \text{for all cubes} \,\, Q.
$$
When $s=\infty$, $\mathsf{RH}_\infty$ denotes the collection of weights $w$ such that
$$
w(x)\leqslant C \thinspace \dashint_Qw \qquad \text{for all cubes} \,\, Q \text{ and almost every } x\in Q.
$$

In \cite{CN1995,StW1985}, the authors showed that there is an explicit connection between the Reverse H\"older class and the Muckenhoupt class of weights.
\begin{theorem}\label{RH}
    $w\in\mathsf{RH}_{s}$ if and only if $w^{s}\in\mathsf{A}_\infty$.
\end{theorem}

\hfill

%-------------------------------------------------
% NEW SECTION: MAIN RESULTS
%-------------------------------------------------
\section{Main Results}\label{main}

For each vector exponent $\vec{p}\coloneqq(p_1,p_2,q)$, we define $\Lambda_{\vec{p}}$ as the set of all vector indices $\overrightarrow m\coloneqq(m_1,m_2)\in\mathbb{R}^2$ that satisfy the following conditions:
\begin{enumerate}[\hspace{5mm}i)]
    \item $1\leqslant m_i\leqslant p_i$ for $i=1,2$.
    \item There exists $1\leqslant m\leqslant\infty$ such that $\frac{m_1}{p_1}+\frac{m_2}{p_2}=1+\frac{1}{mq}$.
\end{enumerate}
Observe that $(p_1,p_2)\notin\Lambda_{\vec{p}}$ when $q>1$. From now on, we shall refer to $m$ as the solution for the equation $\frac{m_1}{p_1}+\frac{m_2}{p_2}=1+\frac{1}{mq}$.

\hfill

\begin{figure}[h]
    \begin{tikzpicture}
    \draw[black, thin,->] (-5mm,0mm) -- (90mm,0mm) node [black,thick,anchor=west] {$m_1$};
    \draw[black, thin,->] (0mm,-5mm) -- (0mm,67mm) node [black,thick,anchor=south] {$m_2$};
    \draw[black, thin,dotted] (10mm,67mm) -- (10mm,0mm) node [black,thick,anchor=north] {$1$};
    \draw[black, thin,dotted] (90mm,10mm) -- (0mm,10mm) node [black,thick,anchor=east] {$1$};
    \draw[black, thin,dotted] (40mm,67mm) -- (40mm,0mm) node [black,thick,anchor=north] {$p_1$};
    \draw[black, thin,dotted] (90mm,30mm) -- (0mm,30mm) node [black,thick,anchor=east] {$p_2$};
    % (1,1)
    \filldraw (10mm,10mm) circle (0.5mm) node[black,thin,anchor=north east]{\small $(1,1)$};
    % (p1,p2)
    \filldraw (40mm,30mm) circle (0.5mm) node[black,thin,anchor=south west]{\small $(p_1,p_2)$};
    % (l1)
    \draw[black, thin,dotted] (0mm,17.5mm) -- (23.33mm,0mm);
    \draw[black, thin,dotted] (23.33mm,0mm) -- (23.33mm,0mm) node [black,thick,anchor=north] {$\frac{p_1}{p}$};
    \draw[black, thin,dotted] (0mm,17.5mm) -- (0mm,17.5mm) node [black,thick,anchor=east] {$\frac{p_2}{p}$};
    \draw[white, thin] (17mm,5mm) -- (17mm,5mm) node[black,anchor=west]{\small $(l_1)$};
    % (l2)
    \draw[black, thin,dotted] (0mm,30mm) -- (40mm,0mm);
    \draw[black, thin,dotted] (40mm,0mm) -- (40mm,0mm);
    \draw[black, thin,dotted] (0mm,30mm) -- (0mm,30mm);
    \draw[white, thin] (19mm,15mm) -- (19mm,15mm) node[black,anchor=east]{\small $(l_2)$};
    % (l3)
    \draw[black, thin,dotted] (0mm,51mm) -- (68mm,0mm);
    \draw[black, thin,dotted] (68mm,0mm) -- (68mm,0mm) node [black,thick,anchor=north] {$p_1\left(1+\frac1q\right)$};
    \draw[black, thin,dotted] (0mm,51mm) -- (0mm,51mm) node [black,thick,anchor=east] {$p_2\left(1+\frac1q\right)$};
    \draw[white, thin] (16mm,40mm) -- (16mm,40mm) node[black,anchor=west]{\small $(l_3)$};
    % (l4)
    \draw[black, thin,dotted] (0mm,60mm) -- (80mm,0mm);
    \draw[black, thin,dotted] (80mm,0mm) -- (80mm,0mm) node [black,thick,anchor=north west] {$2p_1$};
    \draw[black, thin,dotted] (0mm,60mm) -- (0mm,60mm) node [black,thick,anchor=east] {$2p_2$};
    \draw[white, thin] (16mm,49mm) -- (16mm,49mm) node[black,anchor=west]{\small $(l_4)$};
    % (Gamma set)
    \draw[black, very thick] (10mm,22.5mm) -- (26.67mm,10mm) -- (40mm,10mm) -- (40mm,21mm) -- (28mm,30mm) -- (10mm,30mm) -- (10mm,22.5mm);
    \filldraw[gray, opacity = 0.5] (10mm,22.5mm) -- (26.67mm,10mm) -- (40mm,10mm) -- (40mm,21mm) -- (28mm,30mm) -- (10mm,30mm) -- (10mm,22.5mm);
    \draw[white, thin] (27mm,21mm) -- (27mm,21mm) node[black]{\huge $\Lambda_{\vec p}$};
    % Lines' equations
    \draw[white, thin] (60mm,65mm) -- (60mm,65mm) node[black,anchor=west]{$(l_1):\frac{m_1}{p_1}+\frac{m_2}{p_2}=\frac1p$};
    \draw[white, thin] (60mm,59mm) -- (60mm,59mm) node[black,anchor=west]{$(l_2):\frac{m_1}{p_1}+\frac{m_2}{p_2}=1$};
    \draw[white, thin] (60mm,53mm) -- (60mm,53mm) node[black,anchor=west]{$(l_3):\frac{m_1}{p_1}+\frac{m_2}{p_2}=1+\frac1q$};
    \draw[white, thin] (60mm,47mm) -- (60mm,47mm) node[black,anchor=west]{$(l_4):\frac{m_1}{p_1}+\frac{m_2}{p_2}=2$};
    \end{tikzpicture}
    \caption{A visualization for $\Lambda_{\vec p}$ when $1<p<q$ and $1<p_1,p_2<\frac{n}{\al}$. The shape of $\Lambda_{\vec p}$ would change depending on the relative positions between $1,p,p_1,p_2,q$ and $\frac{n}{\al}$ on the real number line. \label{IndiceDomain}}
\end{figure}

For simplicity and clarity, we first state our results in 1-vector-weight setting: when $u^\frac1q=v_1^\frac{1}{p_1}v_2^\frac{1}{p_2}$ and $\frac1q=\frac{1}{p_1}+\frac{1}{p_2}-\frac\al n$. In this scenario, we will utilize the commonly used exponent scale $v_i=w_i^{p_i}$ for $i=1,2$.

\begin{theorem}\label{1VectorWeight}
Suppose $p_1,p_2>1$, $\frac1q=\frac{1}{p_1}+\frac{1}{p_2}-\frac\al n$. If $1\leqslant p\leqslant q$ or $p<1<q\leqslant\frac{p}{1-p}$ or $p\leqslant q\leqslant1$, and
\begin{align}\label{1VectorWeightCondition}
    (w_1,w_2)\in\bigcup_{\overrightarrow m \in \Lambda_{\vec p}}\A_{\left[\frac{p_1}{p_1-m_1+1},\,\frac{p_2}{p_2-m_2+1}\right],\,qm'}\,,
\end{align}
then
\begin{align}\label{BIbdd1w}
    \|\mathsf{BI}_{\alpha}(f,g)\|_{L^q(w_1^qw_2^q)} \lesssim \|f\|_{L^{p_1}(w_1^{p_1})} \|g\|_{L^{p_2}(w_2^{p_2})}.
\end{align}
\end{theorem}

\hfill

\begin{remark}
It would be interesting to see what the condition \eqref{1VectorWeightCondition} looks like in the case of power weights. Let $w_1(x)=|x|^A$ and $w_2(x)=|x|^B$, we have $(|x|^A,|x|^B)\in\A_{\left[\frac{p_1}{p_1-m_1+1},\,\frac{p_2}{p_2-m_2+1}\right],\,qm'}$ if and only if
    \begin{align*}
        M=\sup_Q \left(\dashint_Q |x|^{(A+B)qm'}dx\right)^{\frac{1}{qm'}}\left(\dashint_Q{|x|^{\frac{Ap_1}{1-m_1}}}dx\right)^{\frac{m_1-1}{p_1}} \left(\dashint_Q{|x|^{\frac{Bp_2}{1-m_2}}}dx\right)^{\frac{m_2-1}{p_2}} <\infty
    \end{align*}
    For every cube $Q\in\mathbb{R}^n$, if $|c_Q|>2\ell(Q)$, then $|x|\sim|c_Q|$ for all $x\in Q$, and hence
    $$M\approx|C_Q|^{A+B-A-B}=1.$$
    On the other hand, if $|c_Q|\leqslant2\ell(Q)$, then $Q\subset B(\mathsf{O},3\ell(Q))$, and by using polar coordinates in $\mathbb{R}^n$ have
    \begin{align*}
        M & \lesssim\sup_Q\,\ell(Q)^{\frac{-n}{qm'}+\frac{-n(m_1-1)}{p_1}+\frac{-n(m_2-1)}{p_2}} \left(\int_0^{3\ell(Q)}r^{(A+B)qm'}r^{n-1}dr\right)^{\frac{1}{qm'}} \\
        & \hspace{25mm} \times \left(\int_0^{3\ell(Q)}r^{\frac{Ap_1}{1-m_1}}r^{n-1}dr\right)^{\frac{m_1-1}{p_1}} \left(\int_0^{3\ell(Q)}r^{\frac{Bp_2}{1-m_2}}r^{n-1}dr\right)^{\frac{m_2-1}{p_2}} \\
        & \approx \sup_Q\,\ell(Q)^{\frac{-n}{qm'}+\frac{-n(m_1-1)}{p_1}+\frac{-n(m_2-1)}{p_2}+A+B+\frac{n}{qm'}-A+\frac{n(m_1-1)}{p_1}-B+\frac{n(m_2-1)}{p_2}}=1
    \end{align*}
    provided that $A$ and $B$ simultaneously satisfy the following conditions:
$$\begin{cases}\frac{Ap_1}{1-m_1}+n>0\\\frac{Ap_2}{1-m_2}+n>0\\(A+B)qm'+n>0.\end{cases}$$
Conversely, if any of the above conditions is violated, then one can choose the cube Q that contains the origin to see that $M=\infty$. In other words, we have shown that $(|x|^A,|x|^B)\in\A_{\left[\frac{p_1}{p_1-m_1+1},\,\frac{p_2}{p_2-m_2+1}\right],\,qm'}$ if and only if
$$(A,B)\in\triangle_{\overrightarrow m,\vec p}\coloneqq \left\{A<\frac{n(m_1-1)}{p_1} \, \text{ and } \, B<\frac{n(m_2-1)}{p_2} \, \text{ and } \, A+B>\frac{-n}{qm'}\right\}.$$
Therefore, our condition \eqref{1VectorWeightCondition} is equivalent to
$$(A,B)\in\bigcup_{\overrightarrow m \in \Lambda_{\vec p}}\triangle_{\overrightarrow m,\vec p}$$
whose shape varies depending on the relative positions between $1,p,p_1,p_2,q$ and $\frac{n}{\al}$ on the real number line. An example of such shapes is shown in Figure \ref{Compare}.
\end{remark}

\begin{remark}
The result stated in Theorem \ref{1VectorWeight} applies to a more general class of weights, not just the power weights. However, our result is sufficient but not necessary in general. Komori-Furuya \cite{Komori2019} proved that the necessary and sufficient condition for \eqref{BIbdd1w} when $w_1(x)=|x|^A$ and $w_2(x)=|x|^B$ is:
\begin{align} \label{Komori}
    \begin{cases}A<\frac{n}{p_1'} \, \text{ and } \, A\leqslant n-\al \\ B<\frac{n}{p_2'} \, \text{ and } \, B\leqslant n-\al \\ A+B>\frac{-n}{q} \, \text{ and } \, A+B\geqslant \al-n \end{cases}
\end{align}
while \eqref{1VectorWeightCondition} restricted to these power weights gives a smaller domain for $(A,B)$, see Figure \ref{Compare}, since we have:
$$\frac{n(m_i-1)}{p_i}\leqslant\frac{n}{p_i'} \, , \, \frac{-n}{qm'}\geqslant\frac{-n}{q} \, ,$$
and $$\frac{n(m_i-1)}{p_i}<n-\al \, , \, \frac{-n}{qm'}>\al-n \, ,$$
because $\frac1q=\frac{1}{p_1}+\frac{1}{p_2}-\frac\al n$ and $\frac{m_1}{p_1}+\frac{m_2}{p_2}=1+\frac{1}{mq}$ implies:
$$\frac{n(m_1-1)}{p_1}+\frac{n(m_2-1)}{p_2}=n-\al-\frac{n}{qm'}.$$
\end{remark}

\begin{figure}[h]
    \begin{tikzpicture}[scale=0.9]
    \draw[black, thin,->] (-70mm,0mm) -- (60mm,0mm) node [black,thick,anchor=west] {$A$};
    \draw[black, thin,->] (0mm,-80mm) -- (0mm,55mm) node [black,thick,anchor=south] {$B$};
    \draw[black, thin, dotted] (0mm,50mm) -- (50mm,0mm) node [black,thick,anchor=north west] {\small $n-\al$};
    \draw[black, thin, dotted] (0mm,50mm) -- (0mm,50mm) node [black,thick,anchor=east] {\small $n-\al$};
    \draw[black, thin] (40mm,0mm) -- (40mm,0mm) node [black,thick,anchor=north west] {\small $\frac{n}{p_1'}$};
    \draw[black, thin] (0mm,30mm) -- (0mm,30mm) node [black,thick,anchor=south east] {\small $\frac{n}{p_2'}$};
    \draw[black, thin,dotted] (10mm,-50mm) -- (10mm,0mm) node [black,thick,anchor=south] {\small $\frac{n}{p'}$};
    \draw[black, thin,dotted] (-50mm,10mm) -- (0mm,10mm) node [black,thick,anchor=west] {\small $\frac{n}{p'}$};
    \draw[black, thin,dotted] (-50mm,10mm) -- (-50mm,0mm) node [black,thick,anchor=north east] {\small $\al-n$};
    \draw[black, thin] (-40mm,0mm) -- (-40mm,0mm) node [black,thick,anchor=north east] {\small $-\frac{n}{q}$};
    \draw[black, thin] (0mm,-40mm) -- (0mm,-40mm) node [black,thick,anchor=north east] {\small $-\frac{n}{q}$};
    \draw[black, thin,dotted] (10mm,-50mm) -- (0mm,-50mm) node [black,thick,anchor=east] {\small $\al-n$};
    % Our Range for (A,B)
    \filldraw[gray, opacity=0.5] (-50mm,10mm) -- (10mm,-50mm) -- (40mm,-50mm) -- (40mm,10mm) -- (20mm,30mm) -- (-50mm,30mm) -- (-50mm,10mm);
    % Komori's Range for (A,B)
    \draw[black,very thick,dashed] (40mm,-80mm) -- (-70mm,30mm) -- (40mm,30mm) -- (40mm,-80mm);
    \end{tikzpicture}
    \caption{When $1<p<q$ and $1<p_1,p_2<\frac{n}{\al}$, the shaded region represents the domain given by \eqref{1VectorWeightCondition}; while the region enclosed by the dashed triangle represents the domain given by \eqref{Komori}. \label{Compare}}
\end{figure}

In 2-vector weight settings, we no longer have $u^\frac1q=v_1^\frac{1}{p_1}v_2^\frac{1}{p_2}$ nor $\frac1q=\frac{1}{p_1}+\frac{1}{p_2}-\frac\al n$, and things start to get more complicated. To state our new results, we need to introduce some new notations and concepts. For each $\overrightarrow m \in \Lambda_{\vec{p}}$, let $\mathscr{Y}_{\overrightarrow m,\vec{p}}$ be the collection of all vector Young functions $\vec \phi=(\phi_1,\phi_2,\phi)$ that satisfy the conditions: $\frac{t^{m_i/p_i}}{\phi_i^{-1}(t)}$ is the inverse of a Young function that belongs to $B_{p_i,\frac{p_iq}{p}}$ for $i=1,2$. Also, let $\mathscr{Y}^*_{\overrightarrow m,\vec{p}} \subset \mathscr{Y}_{\overrightarrow m,\vec{p}}$ be such that the Young function $\phi$ in $\vec \phi$ satisfies the condition: $\frac{t^{1/(mq)'}}{\phi^{-1}(t)}$ is the inverse of a Young function that belongs to $B_{q'}$.

\hfill

For each vector exponent $\vec p$ and each vector Young function $\vec \phi=(\phi_1,\phi_2,\phi)$, let $\A^{\vec\phi}_{\vec p}$ denote the collection of all triple weights $(u,v_1,v_2)$ that satisfy the following Muckenhoupt-type condition:
\begin{align}\label{MuckenhouptypeCondition}
    \sup_Q |Q|^{\frac{\al}{n}+\frac{1}{q}-\frac{1}{p}} \|u^{\frac{1}{q}}\|_{\phi,Q} \, \|v_1^{-\frac{1}{p_1}}\|_{\phi_1,Q} \, \|v_2^{-\frac{1}{p_2}}\|_{\phi_2,Q} < \infty.
\end{align}

When $\phi(t)=t^r$ with $r>1$, we have $\A^{\vec\phi}_{\vec p}=\A^{(\phi_1,\,\phi_2,\,t^r)}_{\vec p}$ denote the collection of all triple weights $(u,v_1,v_2)$ that satisfy 
\begin{align}\label{MuckenhouptypeConditionNOBUMPonU}
    \sup_Q |Q|^{\frac{\al}{n}+\frac{1}{q}-\frac{1}{p}} \|u^\frac{1}{q}\|_{L^r,Q} \, \|v_1^{-\frac{1}{p_1}}\|_{\phi_1,Q} \, \|v_2^{-\frac{1}{p_2}}\|_{\phi_2,Q} < \infty.
\end{align}

\hfill

In the following, we state our 2-weight results. The first result gives sufficient conditions for the weighted boundedness of $\mathsf{BI}_\al$ in the case $1\leqslant p\leqslant q$. Our proposed conditions extend the most recently known condition in \cite{KC2017}.

\begin{theorem}\label{1pq}
Given $p_1,p_2>1$ and $1\leq p\leq q$, if
\begin{align}\label{UnionCondition1pq}
    (u,v_1,v_2)\in \bigcup_{\substack{\overrightarrow m \, \in  \,\Lambda_{\vec p} \\ \vec \phi \, \in \, \mathscr{Y}^*_{\overrightarrow m,\vec p}}}\A^{\vec \phi}_{\vec p}
\end{align}
then
\begin{align*}
    \|\mathsf{BI}_{\alpha}(f,g)\|_{L^q(u)} \lesssim \|f\|_{L^{p_1}(v_1)} \|g\|_{L^{p_2}(v_2)}.
\end{align*}
\end{theorem}

Our second result in 2-weight settings gives sufficient conditions for the weighted boundedness of $\mathsf{BI}_\al$ in the case $p\leqslant q\leqslant 1$. Our proposed conditions extend the  most recently known condition in \cite{Kabe2014}.

\begin{theorem}\label{pq1}
Given $p_1,p_2>1$ and $p\leqslant q\leqslant1$, if
\begin{align}\label{UnionConditionpq1}
    (u,v_1,v_2)\in \bigcup_{\substack{\overrightarrow m \, \in \, \Lambda_{\vec p} \\ (\phi_1\,,\,\phi_2\,,\,t^{qm'}) \, \in \, \mathscr{Y}_{\overrightarrow m,\vec p}}} \A^{(\phi_1,\,\phi_2,\,t^{qm'})}_{\vec p}
\end{align}
then
\begin{align*}
    \|\mathsf{BI}_{\alpha}(f,g)\|_{L^q(u)} \lesssim \|f\|_{L^{p_1}(v_1)} \|g\|_{L^{p_2}(v_2)}.
\end{align*}
\end{theorem}

The next theorem gives sufficient conditions for the weighted boundedness of $\mathsf{BI}_\al$ in the case when $p<1<q$. This is a completely new result in the study of the operator. In this case, we have found an upper restriction for the range of $q$; that is $q\leqslant\frac{p}{1-p}$. The reason for this restriction will be explained clearly in the proof of the theorems.

\begin{theorem}\label{p1q}
Given $p_1,p_2>1$ and $p<1<q\leqslant\frac{p}{1-p}$, if
\begin{align*}
    (u,v_1,v_2)\in \bigcup_{\substack{\overrightarrow m \, \in  \,\Lambda_{\vec p} \\ \vec \phi \, \in \, \mathscr{Y}^*_{\overrightarrow m,\vec p}}}\A^{\vec \phi}_{\vec p}
\end{align*}
then
\begin{align*}
    \|\mathsf{BI}_{\alpha}(f,g)\|_{L^q(u)} \lesssim \|f\|_{L^{p_1}(v_1)} \|g\|_{L^{p_2}(v_2)}.
\end{align*}
\end{theorem}

\begin{remark}\label{powerbumpremark}
    In both Theorems \ref{1pq} and \ref{p1q}, and like-wise for Theorem \ref{pq1}, if we choose $\phi$ and $\phi_i$ to be power-bumps, for instance
    \begin{align*}
        \phi(t)&=t^{q(m'+\delta)} \\
        \phi_i(t)&=t^{p_i(m_i'-1+\delta)}
    \end{align*}
with arbitrarily small $\delta>0$, then the inverses of $\frac{t^{1/(mq)'}}{\phi^{-1}}$ and $\frac{t^{m_i/p_i}}{\phi_i^{-1}}$ are respectively the Young functions:
    \begin{align*}
        \psi(t) & = t^\frac{(mq)'(qm'+\delta)}{qm'+\delta-(mq)'} \in B_{q'}, \\
        \psi_i(t) & = t^\frac{p_i(m_i'+m_i\delta)'}{m_i} \in B_{p_i}\subset B_{p_i,\frac{p_iq}{p}},
    \end{align*}
and condition \eqref{MuckenhouptypeCondition} would become:
    $$\sup_Q |Q|^{\frac{\al}{n}+\frac{1}{q}-\frac{1}{p}} \|u^{\frac{1}{q}}\|_{L^{q(m'+\delta)},Q} \, \|v_1^{-\frac{1}{p_1}}\|_{L^{p_1(m_1'-1+\delta)},Q} \, \|v_2^{-\frac{1}{p_2}}\|_{L^{p_2(m_2'-1+\delta)},Q} < \infty.$$
Since $\delta$ in each of the Young functions could have been chosen arbitrarily, one could have manipulated the exponents to obtain the following equivalent condition:
    $$\sup_Q |Q|^{\frac{\al}{n}+\frac{1}{q}-\frac{1}{p}} \|u^{\frac{1}{q}}\|_{L^{qm'+\delta},Q} \, \|v_1^{-\frac{1}{p_1}}\|_{L^{p_1(m_1'-1)+\delta},Q} \, \|v_2^{-\frac{1}{p_2}}\|_{L^{p_2(m_2'-1)+\delta},Q} < \infty.$$
\end{remark}

\begin{remark}\label{logbumpremark}
    One could significantly reduce the size of the left-hand side in \eqref{MuckenhouptypeCondition} by replacing the power-bumps with the log-bumps, for instance
    \begin{align*}
        \phi(t)&=t^{qm'}\log(1+t)^{(q-1)m'+\delta} \\
        \phi_i(t)&=t^{p_i(m_i'-1)}\log(1+t)^{m_i'-1+\delta}
    \end{align*}
    with arbitrarily small $\delta>0$, then the inverses of $\frac{t^{1/(mq)'}}{\phi^{-1}}$ and $\frac{t^{m_i/p_i}}{\phi_i^{-1}}$ are respectively the Young functions:
    \begin{align*}
        \psi(t) & = t^{q'}\log(1+t)^{-1-\frac{\delta}{(q-1)m'}} \in B_{q'}, \\
        \psi_i(t) & = t^{p_i}\log(1+t)^{-1-(m_i-1)\delta} \in B_{p_i}\subset B_{p_i,\frac{p_iq}{p}},
    \end{align*}
which help improve condition \eqref{MuckenhouptypeCondition}.
\end{remark}

We note here that: one can make the functions $\phi,\phi_1,\phi_2$ arbitrarily smaller than the above mentioned log-bumps, for example considering the $\log(1+\log(1+...))$ functions, to keep improving condition \eqref{MuckenhouptypeCondition} infinitely much more. On top of that, the introduction of the triples $(m,m_1,m_2)$ allows us to take into account even more possible weights $u,v_1,v_2$. For instance, in Theorem \ref{1pq}, if we chose $m_i$ so that $\frac{m_1}{p_1}+\frac{m_2}{p_2}$ is close to $1+\frac{1}{q}$, we may have at least one $\phi_i$ close to $t^{p_i'}$, but $\phi$ will be close to $t^\infty$ (yielding $\infty$-norm) as a compensation. This is good in the sense that in case $u$ is very nice, then we can impose a very strong norm on $u$ while requiring weaker norms on $v_i$. Since the choice of $m_i$ is ours, we can suppress any of the weights $u,v_1,v_2$ to leave more rooms for the others. Therefore, we have obtained the largest class of weights for the boundedness of $\mathsf{BI}_\al$ that ever appeared in the literature. Below, we will discuss this in more details by comparing our new results with the most recently known results. For the purpose of clarity and simplicity, all results shall be discussed in their corresponding power-bump settings. \\

Looking at the last condition in Remark \ref{powerbumpremark}, if we choose $m_1=\frac{p_1}{r}$ and $m_2=\frac{p_2}{s}$ where $p_1>r>1$, $p_2>s>1$ and $\frac1r+\frac1s=1$, that condition becomes
\begin{align*}
    & \sup_Q |Q|^{\frac{\al}{n}+\frac{1}{q}-\frac{1}{p}} \|u^{\frac{1}{q}}\|_{L^{q+\delta},Q} \, \|v_1^{-\frac{1}{p_1}}\|_{L^{\frac{p_1r}{p_1-r}+\delta},Q} \, \|v_2^{-\frac{1}{p_2}}\|_{L^{\frac{p_2s}{p_2-s}+\delta},Q} < \infty \\
    \iff & \sup_Q |Q|^{\frac{\al}{n}+\frac{1}{q}-\frac{1}{p}} \|u^{\frac{1}{q}}\|_{L^{q+\delta},Q} \, \|v_1^{-\frac{r}{p_1}}\|_{L^{\frac{p_1}{p_1-r}+\frac\delta r},Q}^\frac{1}{r} \, \|v_2^{-\frac{s}{p_2}}\|_{L^{\frac{p_2}{p_2-s}+\frac\delta s},Q}^\frac1s < \infty.
\end{align*}
Let $\psi(t)=t^{q+\delta}$, $\phi_1(t)=t^{\frac{p_1}{p_1-r}+\frac\delta r}$ and $\phi_2(t)=t^{\frac{p_2}{p_2-s}+\frac\delta s}$, then straightforward computations give $\overline{\psi}\in B_{q'}$, $\overline{\phi_1}\in B_{\frac{p_1}{r}}$ and $\overline{\phi_2}\in B_{\frac{p_2}{s}}$, which is exactly the same condition of Theorem 2.2 in \cite{KC2017}. Of course, our Theorem \ref{1pq} is much more general because of more possible choices for $m_1$ and $m_2$, not to mention that our extra-integrability conditions $B_{p_i,\frac{p_iq}{q}}$ are better than the $B_{p_i}$ conditions. We note here that there is a slight difference between the Young functions in \cite{KC2017} and ours. Our $\phi_1$ and $\phi_2$ are, in fact, their $\phi_1$ and $\phi_2$ respectively composed with $(\cdot)^r$ and $(\cdot)^s$ inside, and these differences make their conditions $B_{\frac{p_1}{r}}$ and $B_{\frac{p_2}{s}}$ scale into $B_{p_1}$ and $B_{p_2}$ in the context of this paper.\\

For Theorem \ref{pq1}, if we choose $m_1=p_1$ and $m_2=p_2$, by a similar argument as in Remark \ref{powerbumpremark}, condition \eqref{MuckenhouptypeConditionNOBUMPonU} becomes
\begin{align*}
    & \sup_Q |Q|^{\frac{\al}{n}+\frac{1}{q}-\frac{1}{p}} \|u^{\frac{1}{q}}\|_{L^{\frac{q}{1-q}},Q} \, \|v_1^{-\frac{1}{p_1}}\|_{L^{p_1'+\delta_1},Q} \, \|v_2^{-\frac{1}{p_2}}\|_{L^{p_2'+\delta_2},Q} < \infty \\
    \iff & \sup_Q |Q|^{\frac{\al}{n}+\frac{1}{q}-\frac{1}{p}} \left(\dashint_Qu^\frac{1}{1-q}\right)^{\frac{1-q}{q}} \left(\dashint_Qv_1^{-\frac{r_1p_1'}{p_1}}\right)^{\frac{1}{r_1p_1'}} \left(\dashint_Qv_2^{-\frac{r_2p_2'}{p_2}}\right)^{\frac{1}{r_2p_2'}} < \infty
\end{align*}
where $r_i=1+\frac{\delta_i}{p_i'}$. By adopting the scales $u\rightarrow u^q$, $v_i\rightarrow v_i^{p_i}$ and manipulating the $\delta_i$'s so that $r_1=r_2=r$, we end up having
$$\sup_Q |Q|^{\frac{\al}{n}+\frac{1}{q}-\frac{1}{p}} \left(\dashint_Qu^\frac{q}{1-q}\right)^{\frac{1-q}{q}} \left(\dashint_Qv_1^{-rp_1'}\right)^{\frac{1}{rp_1'}} \left(\dashint_Qv_2^{-rp_2'}\right)^{\frac{1}{rp_2'}} < \infty$$
which gives the same condition of Theorem 1.1 in \cite{Kabe2014}. 

\hfill

\begin{proof}[Proof of Theorem \ref{1VectorWeight}]
The proposed condition in the theorem implies that the pair of weights $(w_1,w_2)\in\mathsf{A}_{\left[\frac{p_1}{p_1-m_1+1},\,\frac{p_2}{p_2-m_2+1}\right],\,qm'}$, for some triple $(m,m_1,m_2)$. Since $\left(\frac{p_i}{p_i-m_i+1}\right)'=p_i(m_i'-1)$, by Theorem \ref{1wApqProperty}, we have
$$(w_1w_2)^{qm'}\in \mathsf{A}_{2qm'}\subset\mathsf{A}_\infty \qquad \text{and} \qquad w_i^{-p_i(m_i'-1)}\in \mathsf{A}_{2p_i(m_i'-1)}\subset\mathsf{A}_\infty.$$

By Lemma \ref{AinfinityProperty}, with appropriate re-scaling on the exponents, there exists a number $\delta>0$ such that
\begin{align*}
    \left(\dashint_Q(w_1w_2)^{qm'+\delta}\right)^\frac{1}{qm'+\delta} & \lesssim \left(\dashint_Qw_1^{qm'}w_2^{qm'}\right)^\frac{1}{qm'}, \\
    \left(\dashint_Qw_i^{-[p_i(m_i'-1)+\delta]}\right)^\frac{1}{p_i(m_i'-1)+\delta} & \lesssim \left(\dashint_Qw_i^{-p_i(m_i'-1)}\right)^\frac{1}{p_i(m_i'-1)}.
\end{align*}
Let $u=w_1^qw_2^q$ and $v_i=w_i^{p_i}$, then by Remark \ref{powerbumpremark} and Theorems \ref{1pq}, \ref{p1q}, \ref{pq1}, we obtain the weighted boundedness for $\mathsf{BI}_\al$. 
\end{proof}

Our results yields immediate applications for the natural fractional maximal function associated to $\mathsf{BI}_{\alpha}$, namely

$$\mathsf{BM}_{\alpha}(f,g)(x)=\sup_{Q\ni x}|Q|^{\frac{\al}{n}-1}\int_{|y|_\infty\leqslant\ell(Q)}|f(x-y)\,g(x+y)|\,dy.$$

It has been known from \cite{DL2002} that
$$\mathsf{BM}_{\alpha}(f,g)(x)\leqslant C\,\mathsf{BI}_{\alpha}(f,g)(x),$$
for $0<\al<n$. Therefore, we have the following corollaries.

\begin{cor}
Given $p_1,p_2>1$ and $1\leq p\leq q$, if
\begin{align*}
    (u,v_1,v_2)\in \bigcup_{\substack{\overrightarrow m \, \in  \,\Lambda_{\vec p} \\ \vec \phi \, \in \, \mathscr{Y}^*_{\overrightarrow m,\vec p}}} \A^{\vec\phi}_{\vec p}
\end{align*}
then
\begin{align*}
    \|\mathsf{BM}_{\alpha}(f,g)\|_{L^q(u)} \lesssim \|f\|_{L^{p_1}(v_1)} \|g\|_{L^{p_2}(v_2)}.
\end{align*}
\end{cor}

\begin{cor}
Given $p_1,p_2>1$ and $p\leqslant q\leqslant1$, if
\begin{align*}
    (u,v_1,v_2)\in \bigcup_{\substack{\overrightarrow m \, \in  \,\Lambda_{\vec p} \\ (\phi_1\,,\,\phi_2\,,\,t^{qm'}) \, \in \, \mathscr{Y}_{\overrightarrow m,\vec p}}} \A^{(\phi_1\,,\,\phi_2\,,\,t^{qm'})}_{\vec p}
\end{align*}
then
\begin{align*}
    \|\mathsf{BM}_{\alpha}(f,g)\|_{L^q(u)} \lesssim \|f\|_{L^{p_1}(v_1)} \|g\|_{L^{p_2}(v_2)}.
\end{align*}
\end{cor}

\begin{cor}
Given $p_1,p_2>1$ and $p<1<q\leqslant\frac{p}{1-p}$, if
\begin{align*}
    (u,v_1,v_2)\in \bigcup_{\substack{\overrightarrow m \, \in  \,\Lambda_{\vec p} \\ \vec \phi \, \in \, \mathscr{Y}^*_{\overrightarrow m,\vec p}}} \A^{\vec\phi}_{\vec p}
\end{align*}
then
\begin{align*}
    \|\mathsf{BM}_{\alpha}(f,g)\|_{L^q(u)} \lesssim \|f\|_{L^{p_1}(v_1)} \|g\|_{L^{p_2}(v_2)}.
\end{align*}
\end{cor}

\begin{cor}
Suppose $p_1,p_2>1$, $\frac1q=\frac{1}{p_1}+\frac{1}{p_2}-\frac\al n$. If $1\leqslant p\leqslant q$ or $p<1<q\leqslant\frac{p}{1-p}$ or $p\leqslant q\leqslant1$, and 
$$(w_1,w_2)\in\bigcup_{\overrightarrow m \in \Lambda_{\vec p}}\A_{\left[\frac{p_1}{p_1-m_1+1},\,\frac{p_2}{p_2-m_2+1}\right],\,qm'}\,.$$
then
\begin{align*}
    \|\mathsf{BM}_{\alpha}(f,g)\|_{L^q(w_1^qw_2^q)} \lesssim \|f\|_{L^{p_1}(w_1^{p_1})} \|g\|_{L^{p_2}(w_2^{p_2})}.
\end{align*}
\end{cor}

\hfill

%-------------------------------------------------
% NEW SECTION: Proof of the Main Theorems
%-------------------------------------------------
\section{Proof of the Main Theorems}\label{ProofMain}

We will begin with the proof for Theorems \ref{1pq} and \ref{p1q} together, as they are almost completely the same. Their differences will be pointed out clearly along the way we prove them.

\begin{proof}[Proof of Theorems \ref{1pq} and \ref{p1q}]

\hfill

Without loss of generality, we assume that $f$ and $g$ are non-negative $C_c^\infty$-functions on $\R^n$. By duality, for every non-negative function $h\in L^{q'}(\R^n)$, we only need to gain control over the following integral:
\begin{align*}
    \lefteqn{\int_{\mathbb{R}^n} BI_{\alpha}(f,g)(x) \, h(x)\, u(x)^{\frac{1}{q}} dx} \hspace{4cm} \\
    & \lesssim \sum_{Q\in\mathscr{D}}|Q|^{\frac{\alpha}{n}-1} \int_Q \left[(f\1_{3Q})*(g\1_{3Q})\right](2x) \, h(x) \, u(x)^{\frac{1}{q}} \, dx
\end{align*}
where we have utilized the estimate \eqref{ConvolutionDomination}.

Since $p_i>1$, there exist $m_i$ such that $1\leqslant m_i\leqslant p_i$ and $\frac{m_1}{p_1}+\frac{m_2}{p_2}\geqslant1$. Let $m\geqslant1$ be defined via the equation $\frac{m_1}{p_1}+\frac{m_2}{p_2}=1+\frac{1}{mq}$. The existence of $m$ is always possible when $p\geqslant1$ due to the continuity of the function $\frac{m_1}{p_1}+\frac{m_2}{p_2}$ of the two variables $m_1$ and $m_2$. When $p<1$, we then have $1+\frac{1}{mq}=\frac{m_1}{p_1}+\frac{m_2}{p_2}\geqslant\frac{1}{p}$ which implies $mq\leqslant\frac{p}{1-p}$. If we had $q>\frac{p}{1-p}$, then there would be no choice for such $m\geqslant1$  (see Figure \ref{IndiceDomain}, and imagine when the line $(l_3)$ goes below $(l_1)$), which in turns implies the non-existence of the later defined Young function $\phi$. So, when $p<1$, our theorem would only be valid for $q\leqslant\frac{p}{1-p}$.

\hfill

By H\"older's and Young's inequalities, we have
\begin{align*}
    & \int_{\mathbb{R}^n} BI_{\alpha}(f,g)(x) \, h(x)\, u(x)^{\frac{1}{q}} dx \\
    & \leqslant \sum_{Q\in\mathscr{D}} |Q|^{\frac{\alpha}{n}-1} \left[\int_Q \left[(f\1_{3Q})*(g\1_{3Q})\right](2x)^{mq} \, dx\right]^\frac{1}{mq} \left[\int_Q h(x)^{(mq)'} \, u(x)^{\frac{(mq)'}{q}} \, dx\right]^\frac{1}{(mq)'} \\
    & \lesssim \sum_{Q\in\mathscr{D}} |Q|^{\frac{\alpha}{n}-1} \left(\int_{3Q} f^\frac{p_1}{m_1}\right)^\frac{m_1}{p_1} \left(\int_{3Q} g^\frac{p_2}{m_2}\right)^\frac{m_2}{p_2} \left(\int_Q h^{(mq)'} \, u^{\frac{(mq)'}{q}}\right)^\frac{1}{(mq)'} \\
    & \simeq \sum_{Q\in\mathscr{D}} |Q|^{\frac{\alpha}{n}+1} \, \|f\|_{L^\frac{p_1}{m_1},3Q} \, \|g\|_{L^\frac{p_2}{m_2},3Q} \, \|hu^\frac{1}{q}\|_{L^{(mq)'},Q} \\
    & \lesssim \sum_{t=1}^{2^n} \sum_{P\in\mathscr{D}_t}|P|^{\frac{\alpha}{n}+1} \, \|f\|_{L^\frac{p_1}{m_1},P} \, \|g\|_{L^\frac{p_2}{m_2},P} \, \|hu^\frac{1}{q}\|_{L^{(mq)'},P}
\end{align*}
where the last inequality is obtained by utilizing Theorem \ref{1/3trick}: for every cube $3Q$, there exists a $t\in \{0,1/3\}^n$ and a cube $P\in\mathscr{D}_t$ such that $Q\subset P$ and $\ell(P)\leqslant 6\thinspace \ell(Q)$. Since the sizes of $3Q$ and $P$ are comparable, the number of different cubes $3Q$ contained in the same cube $P$ must be finite, and this is the reason for the validity of the last inequality. For simplicity in the later part of the proof, we relabel the cubes $P$ as $Q$ and have
\begin{align*}
    \int_{\mathbb{R}^n} \bigl|BI_{\alpha}(f,g)(x)\bigr| h(x)u(x)^{\frac{1}{q}} dx \lesssim \sum_{t=1}^{2^n} \sum_{Q\in\mathscr{D}_t}|Q|^{\frac{\alpha}{n}+1} \|f\|_{L^\frac{p_1}{m_1},Q} \|g\|_{L^\frac{p_2}{m_2},Q} \|hu^\frac{1}{q}\|_{L^{(mq)'},Q}.
\end{align*}

For each $t$, we are going to dominate the dyadic sum by a sum over a corresponding sparse family of cubes. For every $k\in\mathbb{Z}$, let $\{Q_{k,j}\}_j$ be a collection of disjoint cubes from $\mathscr{D}_t$ that are maximal with respect to
$$\|f\|_{L^\frac{p_1}{m_1},Q}\,\|g\|_{L^\frac{p_2}{m_2},Q} > 2^{k(n+1)\left(\frac{m_1}{p_1}+\frac{m_2}{p_2}\right)}.$$
Define $E_{k,j}=Q_{k,j}\setminus\bigcup_iQ_{k+1,i}$. The family $\{E_{k,j}\}_{k,j}$ is pair-wise disjoint. Furthermore, let $P$ denote an immediate dyadic parent of $Q_{k,j}$, by the maximality of $Q_{k,j}$ and $Q_{k+1,i}$ we have
\begin{align*}
    \lefteqn{\bigl|Q_{k,j}\cap\,\bigcup_iQ_{k+1,i}\bigr| = \sum_{Q_{k+1,i}\,\subseteq\,Q_{k,j}}|Q_{k+1,i}|} \\
    & \leqslant 2^{-(n+1)(k+1)} \sum_{Q_{k+1,i}\,\subseteq\,Q_{k,j}} \left(\int_{Q_{k+1,i}}f^\frac{p_1}{m_1}\right)^\frac{m_1p_2}{m_1p_2+m_2p_1} \left(\int_{Q_{k+1,i}}g^\frac{p_2}{m_2}\right)^\frac{m_2p_1}{m_1p_2+m_2p_1} \\
    & \leqslant 2^{-(n+1)(k+1)} \left(\sum_{\substack{Q_{k+1,i}\\\subseteq\,Q_{k,j}}} \int_{Q_{k+1,i}}f^\frac{p_1}{m_1}\right)^\frac{m_1p_2}{m_1p_2+m_2p_1} \left(\sum_{\substack{Q_{k+1,i}\\\subseteq\,Q_{k,j}}} \int_{Q_{k+1,i}}g^\frac{p_2}{m_2}\right)^\frac{m_2p_1}{m_1p_2+m_2p_1} \\
    & \leqslant 2^{-(n+1)(k+1)} \left[\left(\int_{Q_{k,j}}f^\frac{p_1}{m_1}\right)^\frac{m_1}{p_1} \left(\int_{Q_{k,j}}g^\frac{p_2}{m_2}\right)^\frac{m_2}{p_2}\right]^\frac{p_1p_2}{m_1p_2+m_2p_1} \\
    & \leqslant 2^{-(n+1)(k+1)} \, |P|\left[\left(\dashint_Pf^\frac{p_1}{m_1}\right)^\frac{m_1}{p_1} \left(\dashint_Pg^\frac{p_2}{m_2}\right)^\frac{m_2}{p_2}\right]^\frac{p_1p_2}{m_1p_2+m_2p_1} \\
    & \leqslant 2^{-(n+1)(k+1)}\,2^n|Q_{k,j}|\,2^{k(n+1)}=\frac{1}{2}|Q_{k,j}|.
\end{align*}
This implies $\bigl|Q_{k,j}\bigr|\leqslant2\bigl|E_{k,j}\bigr|$, and hence $\{Q_{k,j}\}_{k,j}\coloneq\mathscr{S}_t$ is a sparse family. For every $k\in\mathbb{Z}$, let
$$C_k=\left\{Q\in\mathscr{D}_t:\thinspace 2^{k(n+1)\left(\frac{m_1}{p_1}+\frac{m_2}{p_2}\right)} < \|f\|_{L^\frac{p_1}{m_1},Q}\,\|g\|_{L^\frac{p_2}{m_2},Q} \leqslant 2^{(k+1)(n+1)\left(\frac{m_1}{p_1}+\frac{m_2}{p_2}\right)}\right\}.$$
Since every $Q\in \mathscr{D}_t$ for which $|Q|^{\frac{\alpha}{n}+1} \, \|f\|_{L^\frac{p_1}{m_1},Q} \, \|g\|_{L^\frac{p_2}{m_2},Q} \, \|hu^\frac{1}{q}\|_{L^{(mq)'},Q}$ is non-zero must be in some $C_k$, and every $Q\in C_k$ is contained in a unique $Q_{k,j}$, we have
\begin{align*} 
    \lefteqn{\sum_{Q\in\mathscr{D}_t}|Q|^{\frac{\alpha}{n}+1} \, \|f\|_{L^\frac{p_1}{m_1},Q} \, \|g\|_{L^\frac{p_2}{m_2},Q} \, \|hu^\frac{1}{q}\|_{L^{(mq)'},Q}} \hspace{16mm} \\
    & \leqslant \sum_{k\in\mathbb{Z}}\sum_{Q\in C_k} |Q|^{\frac{\alpha}{n}+1} \, \|f\|_{L^\frac{p_1}{m_1},Q} \, \|g\|_{L^\frac{p_2}{m_2},Q} \, \|hu^\frac{1}{q}\|_{L^{(mq)'},Q} \\
    & \leqslant \sum_{k\in\mathbb{Z}} 2^{(k+1)(n+1)\left(\frac{m_1}{p_1}+\frac{m_2}{p_2}\right)} \sum_{Q\in C_k} |Q|^{\frac{\alpha}{n}+1} \, \|hu^\frac{1}{q}\|_{L^{(mq)'},Q} \\
    & \leqslant \sum_{k\in\mathbb{Z}} 2^{(k+1)(n+1)\left(\frac{m_1}{p_1}+\frac{m_2}{p_2}\right)} \sum_j \sum_{\substack{Q\in \mathscr{D}_t \\ Q\subseteq Q_{k,j}}} |Q|^{\frac{\alpha}{n}+\frac{1}{mq}} \, \|hu^\frac{1}{q}\1_Q\|_{L^{(mq)'}} \\
    & = \sum_{k\in\mathbb{Z}} 2^{(k+1)(n+1)\left(\frac{m_1}{p_1}+\frac{m_2}{p_2}\right)} \sum_j \sum_{r=0}^{\infty}\sum_{\substack{Q\in \mathscr{D}_t, \thinspace Q\subseteq Q_{k,j} \\ \ell(Q)=2^{-r}\ell(Q_{k,j})}} |Q|^{\frac{\alpha}{n}+\frac{1}{mq}} \, \|hu^\frac{1}{q}\1_Q\|_{L^{(mq)'}} \\
    & \leqslant \sum_{k\in\mathbb{Z}} 2^{(k+1)(n+1)\left(\frac{m_1}{p_1}+\frac{m_2}{p_2}\right)} \sum_j|Q_{k,j}|^{\frac{\alpha}{n}+\frac{1}{mq}} \sum_{r=0}^{\infty}2^{-r\al-\frac{rn}{mq}} \\
    & \hspace{19mm} \times\left(\sum_{\substack{Q\in \mathscr{D}_t, \thinspace Q\subseteq Q_{k,j} \\ \ell(Q)=2^{-r}\ell(Q_{k,j})}} \int_Q(hu^\frac{1}{q})^{(mq)'}\right)^\frac{1}{(mq)'} \left(\sum_{\substack{Q\in \mathscr{D}_t, \thinspace Q\subseteq Q_{k,j} \\ \ell(Q)=2^{-r}\ell(Q_{k,j})}} 1\right)^\frac{1}{mq} \\
    & = \sum_{k\in\mathbb{Z}} 2^{(k+1)(n+1)\left(\frac{m_1}{p_1}+\frac{m_2}{p_2}\right)} \sum_j|Q_{k,j}|^{\frac{\alpha}{n}+1} \|hu^\frac{1}{q}\|_{L^{(mq)'},Q_{k,j}} \sum_{r=0}^{\infty}2^{-r\al} \\
    & \lesssim \sum_{k,j}|Q_{k,j}|^{\frac{\alpha}{n}+1} \|f\|_{L^\frac{p_1}{m_1},Q_{k,j}}\|g\|_{L^\frac{p_2}{m_2},Q_{k,j}}\|hu^\frac{1}{q}\|_{L^{(mq)'},Q_{k,j}} \\
    & = \sum_{Q\in\mathscr{S}_t}|Q|^{\frac{\alpha}{n}+1} \|f\|_{L^\frac{p_1}{m_1},Q}\|g\|_{L^\frac{p_2}{m_2},Q}\|hu^\frac{1}{q}\|_{L^{(mq)'},Q}\,,
\end{align*}
and we have successfully transitioned from a sum on a dyadic grid to a sum on a sparse family of cubes. From now on, we shall refer to this as the going sparse process.

\hfill

Let $\phi,\phi_1,\phi_2$ be the Young functions that comes from condition \eqref{UnionCondition1pq}, and let $\psi$ denote the inverse function of $\frac{t^{1/(mq)'}}{\phi^{-1}(t)}$ and $\psi_i$ denote the inverse function of $\frac{t^{m_i/p_i}}{\phi_i^{-1}(t)}$. We have $\psi\in B_{q'}$ and $\psi_i\in B_{p_i,\frac{p_iq}{p}}$. By the generalized H\"older's inequality for the Orlicz averages and condition \eqref{UnionCondition1pq}, we continue with the following estimates:
\begin{align*}
    & \int_{\mathbb{R}^n} BI_{\alpha}(f,g)(x) \, h(x)\, u(x)^{\frac{1}{q}} dx \\
    & \lesssim \sum_{t=1}^{2^n} \sum_{Q\in\mathscr{S}_t} |Q|^{\frac{\alpha}{n}+1} \, \|fv_1^\frac{1}{p_1}\|_{\psi_1,Q} \, \|v_1^{-\frac{1}{p_1}}\|_{\phi_1,Q} \, \|gv_2^\frac{1}{p_2}\|_{\psi_2,Q} \, \|v_2^{-\frac{1}{p_2}}\|_{\phi_2,Q} \, \|h\|_{\psi,Q} \, \|u^{\frac{1}{q}}\|_{\phi,Q} \\
    & \lesssim \sum_{t=1}^{2^n} \sum_{Q\in\mathscr{S}_t} |Q|^{\frac{1}{p}+\frac{1}{q'}} \, \|fv_1^\frac{1}{p_1}\|_{\psi_1,Q} \, \|gv_2^\frac{1}{p_2}\|_{\psi_2,Q} \, \|h\|_{\psi,Q} \\
    & = \sum_{t=1}^{2^n} \sum_{Q\in\mathscr{S}_t} |Q|^\frac{1}{p_1} \|fv_1^\frac{1}{p_1}\|_{\psi_1,Q} \, |Q|^\frac{1}{p_2} \|gv_2^\frac{1}{p_2}\|_{\psi_2,Q} \, |Q|^\frac{1}{q'} \|h\|_{\psi,Q} \\
    & \lesssim \sum_{t=1}^{2^n} \left[\sum_{Q\in\mathscr{S}_t} |Q|^{\frac{q}{p}-1}\,\|fv_1^\frac{1}{p_1}\|^\frac{p_1q}{p}_{\psi_1,Q} \, |E_Q|\right]^\frac{p}{p_1q} \left[\sum_{Q\in\mathscr{S}_t} |Q|^{\frac{q}{p}-1}\,\|gv_2^\frac{1}{p_2}\|^\frac{p_2q}{p}_{\psi_2,Q} \, |E_Q|\right]^\frac{p}{p_2q} \\
    & \hspace{89mm}\times \left[\sum_{Q\in\mathscr{S}_t}\|h\|^{q'}_{\psi,Q} \, |E_Q|\right]^\frac{1}{q'} \\
    & \lesssim \sum_{t=1}^{2^n} \left[\sum_{Q\in\mathscr{S}_t} \int_{E_Q}M_{\al_1,\psi_1}(fv_1^\frac{1}{p_1})(x)^\frac{p_1q}{p}dx\right]^\frac{p}{p_1q} \\
    & \hspace{19mm}\times \left[\sum_{Q\in\mathscr{S}_t} \int_{E_Q}M_{\al_2,\psi_2}(gv_2^\frac{1}{p_2})(x)^\frac{p_2q}{p}dx\right]^\frac{p}{p_2q} \left[\sum_{Q\in\mathscr{S}_t}\int_{E_Q}M_{\psi}(h)(x)^{q'}dx\right]^\frac{1}{q'} \\
    & \leqslant \sum_{t=1}^{2^n} \left(\int_{\R^n} M_{\al_1,\psi_1}(fv_1^\frac{1}{p_1})(x)^\frac{p_1q}{p}\,dx\right)^\frac{p}{p_1q} \, \left(\int_{\R^n} M_{\al_2,\psi_2}(gv_2^\frac{1}{p_2})(x)^\frac{p_2q}{p}\,dx\right)^\frac{p}{p_2q} \\
    & \hspace{9cm} \times \left(\int_{\R^n}M_\psi(h)(x)^{q'}\right)^\frac{1}{q'} \\
    & \lesssim \|f\|_{L^{p_1}(v_1)} \, \|g\|_{L^{p_2}(v_2)} \, \|h\|_{q'}
\end{align*}
where $0\leqslant\al_i\coloneq n\left(\frac{1}{p_i}-\frac{p}{p_iq}\right)<n$, and we have applied either Theorem \ref{Bp} or Theorem \ref{Bpq} for each of the three maximal functions to obtain the last inequality. Since $q>1$, by duality we obtain the desired weighted bound for $\mathsf{BI_\al}$.
\end{proof}

\begin{proof}[Proof of Theorems \ref{pq1}]

\hfill

Again, we may assume that $f$ and $g$ are non-negative $C_c^\infty$-functions. Since $p_i>1$, there exist $m_i$ such that $1\leqslant m_i\leqslant p_i$. This leads to the existence of $m\geqslant1$ defined by $\frac{m_1}{p_1}+\frac{m_2}{p_2}=1+\frac{1}{mq}$. Because $p\leqslant q\leqslant1$, such choices are always possible. Also, when $q\leqslant1$, we can avoid the duality argument which causes the required extra bump on the weight $u$. By H\"older's and Young's inequalities, we have
\begin{align*}
    \lefteqn{\int_{\mathbb{R}^n} BI_{\alpha}(f,g)(x)^q \, u(x) \, dx} \hspace{3mm} \\
    & \lesssim \sum_{Q\in\mathscr{D}}|Q|^{q\left(\frac{\alpha}{n}-1\right)} \int_Q \left[(f\1_{3Q})*(g\1_{3Q})\right](2x)^q \, u(x) \, dx \\
    & \leqslant \sum_{Q\in\mathscr{D}} |Q|^{q\left(\frac{\alpha}{n}-1\right)} \left(\int_Q \left[(f\1_{3Q})*(g\1_{3Q})\right](2x)^{qm} \, dx\right)^\frac{1}{m} \left(\int_Q u(x)^{m'} \, dx\right)^\frac{1}{m'} \\
    & \lesssim \sum_{Q\in\mathscr{D}} |Q|^{q\left(\frac{\alpha}{n}-1\right)} \left(\int_{3Q} f(x)^{\frac{p_1}{m_1}} \, dx\right)^\frac{qm_1}{p_1} \left(\int_{3Q} g(x)^{\frac{p_2}{m_2}} \, dx\right)^\frac{qm_2}{p_2} \left(\int_Q u(x)^{m'} \, dx\right)^\frac{1}{m'} \\
    & \simeq \sum_{Q\in\mathscr{D}} |Q|^{\frac{q\alpha}{n}+1} \, \|f\|_{L^\frac{p_1}{m_1},3Q}^q \, \|g\|_{L^\frac{p_2}{m_2},3Q}^q \, \|u^\frac{1}{q}\|_{L^{qm'},Q}^q \\
    & \lesssim \sum_{t=1}^{2^n} \sum_{Q\in\mathscr{D}_t}|Q|^{\frac{q\alpha}{n}+1} \, \|f\|_{L^\frac{p_1}{m_1},Q}^q \, \|g\|_{L^\frac{p_2}{m_2},Q}^q \, \|u^\frac{1}{q}\|_{L^{qm'},Q}^q \\
    & \lesssim \sum_{t=1}^{2^n} \sum_{Q\in\mathscr{S}_t}|Q|^{\frac{q\alpha}{n}+1} \, \|f\|_{L^\frac{p_1}{m_1},Q}^q \, \|g\|_{L^\frac{p_2}{m_2},Q}^q \, \|u^\frac{1}{q}\|_{L^{qm'},Q}^q
\end{align*}
where the last inequality is obtained by a similar going sparse process as in our previous proof. Let $\phi_1,\phi_2$ be the Young functions that comes from condition \eqref{UnionConditionpq1}, and let $\psi_i$ denote the inverse function of $\frac{t^{m_i/p_i}}{\phi_i^{-1}(t)}$. We have $\psi_i\in B_{p_i,\frac{p_iq}{p}}$ which help us obtain the following estimates:
\begin{align*}
    \lefteqn{\int_{\mathbb{R}^n} BI_{\alpha}(f,g)(x)^q \, u(x) \, dx} \hspace{5mm} \\
    & \lesssim \sum_{t=1}^{2^n} \sum_{Q\in\mathscr{S}_t} |Q|^{\frac{q\alpha}{n}+1} \, \|fv_1^\frac{1}{p_1}\|_{\psi_1,Q}^q \, \|v_1^{-\frac{1}{p_1}}\|_{\phi_1,Q}^q \, \|gv_2^\frac{1}{p_2}\|_{\psi_2,Q}^q \, \|v_2^{-\frac{1}{p_2}}\|_{\phi_2,Q}^q \, \|u^\frac{1}{q}\|_{L^{qm'},Q}^q \\
    & \lesssim \sum_{t=1}^{2^n} \sum_{Q\in\mathscr{S}_t} |Q|^{\frac{q}{p}} \, \|fv_1^\frac{1}{p_1}\|_{\psi_1,Q}^q \, \|gv_2^\frac{1}{p_2}\|_{\psi_2,Q}^q \\
    & \leqslant \sum_{t=1}^{2^n} \left[\sum_{Q\in\mathscr{S}_t} |Q|^{\frac{q}{p}-1}\,\|fv_1^\frac{1}{p_1}\|^\frac{p_1q}{p}_{\psi_1,Q} \, |E_Q|\right]^\frac{p}{p_1} \left[\sum_{Q\in\mathscr{S}_t} |Q|^{\frac{q}{p}-1}\,\|gv_2^\frac{1}{p_2}\|^\frac{p_2q}{p}_{\psi_2,Q} \, |E_Q|\right]^\frac{p}{p_2} \\
    & \leqslant \sum_{t=1}^{2^n} \left[\sum_{Q\in\mathscr{S}_t} \int_{E_Q}M_{\al_1,\psi_1}(fv_1^\frac{1}{p_1})(x)^\frac{p_1q}{p}\,dx\right]^\frac{p}{p_1} \left[\sum_{Q\in\mathscr{S}_t} M_{\al_2,\psi_2}(gv_2^\frac{1}{p_2})(x)^\frac{p_2q}{p}\,dx\right]^\frac{p}{p_2} \\
    & \leqslant \sum_{t=1}^{2^n} \left(\int_{\R^n} M_{\al_1,\psi_1}(fv_1^\frac{1}{p_1})(x)^\frac{p_1q}{p}\,dx\right)^\frac{p}{p_1} \left(\int_{\R^n} M_{\al_2,\psi_2}(gv_2^\frac{1}{p_2})(x)^\frac{p_2q}{p}\,dx\right)^\frac{p}{p_2} \\
    & \lesssim \|f\|^q_{L^{p_1}(v_1)} \, \|g\|^q_{L^{p_2}(v_2)}
\end{align*}
where $\al_i\coloneq n\left(\frac{1}{p_i}-\frac{p}{p_iq}\right)$.
\end{proof}

\hfill

%-------------------------------------------------
% NEW SECTION: THE COMMUTATORS
%-------------------------------------------------
\section{The Commutators}\label{Comm}

In this section, we investigate how the commutators would affects our operator $\mathsf{BI}_\al$ and its boundedness conditions. Our findings are new and extends the results in \cite{KC2017} where the commutators on $\mathsf{BI}_\al$ were defined as follows: given a function $b$, the commutator by $b$ with the first component of $\mathsf{BI}_\al$ is defined as
$$
[b,\mathsf{BI_\al}]_1(f,g)=b \thinspace \mathsf{BI_\al}(f,g)-\mathsf{BI_\al}(bf,g),
$$
while the commutator with the second component of $\mathsf{BI}_\al$ is defined by
$$
[b,\mathsf{BI_\al}]_2(f,g)=b \thinspace \mathsf{BI_\al}(f,g)-\mathsf{BI_\al}(f,bg).
$$
If we sequentially apply the commutators by $b_1,...,b_N$ with the first, the second or a mixture of first and second components of $\mathsf{BI}_\al$, we end up getting the general product commutators:
$$
[\vec{b},\mathsf{BI_\al}]_{\vec{\beta}}=[b_N,[b_{N-1}...,[b_2,[b_1,\mathsf{BI_\al}]_{\beta_1}]_{\beta_2}...]_{\beta_{N-1}}]_{\beta_N}
$$
where $\vec{b}=(b_1,...,b_N)$ and $\vec{\beta}=(\beta_1,...,\beta_N)\in\ \{1,2\}^N$. One can prove that
\begin{equation*}
[\sigma(\vec{b}),\mathsf{BI_\al}]_{\sigma(\vec{\beta})}=[\vec{b},\mathsf{BI_\al}]_{\vec{\beta}}
\end{equation*}
where $\sigma$ is any permutation on the N symbols: $1,...,N$. In particular, that is true for $\sigma(\vec{\beta})=(1,...,1,2,...,2)$. Therefore, from now on we will always assume that $\vec{\beta}=(1,...,1,2,...,2)$, and reserve the notation $M$ to denote the number of first component commutators in the general product commutator.

\hfill

Let $\mathsf{K}=\{0,1,...,M\}\times\{0,1,...,N-M\}$. For each $\overrightarrow m \in \Lambda_{\vec{p}}$ and each $\vec k=(k_1,k_2)\in\mathsf{K}$, let $\mathscr{Y}_{\vec k,\overrightarrow m,\vec{p}}$ be the collection of all vector Young functions $\vec \phi=(\phi_1,\phi_2,\phi)$ that satisfy the conditions: $\frac{t^{m_i/p_i}}{\phi_i^{-1}(t)\log(1+t)^{k_i}}$ is the inverse of a Young function that belongs to $B_{p_i,\frac{p_iq}{p}}$ for $i=1,2$. In addition, let $\mathscr{Y}^*_{\vec k,\overrightarrow m,\vec{p}} \subset \mathscr{Y}_{\vec k,\overrightarrow m,\vec{p}}$ be such that the Young function $\phi$ in $\vec \phi$ satisfies the condition: $\frac{t^{1/(mq)'}}{\phi^{-1}(t)\log(1+t)^{N-k_1-k_2}}$ is the inverse of a Young function that belongs to $B_{q'}$. Notice that when $\vec k=(0,0)$, we have $\mathscr{Y}_{(0,0),\overrightarrow m,\vec{p}}=\mathscr{Y}_{\overrightarrow m,\vec{p}}$ but $\mathscr{Y}^*_{(0,0),\overrightarrow m,\vec{p}}\neq\mathscr{Y}^*_{\overrightarrow m,\vec{p}}\,$. Examples for these Young functions are:
\begin{align*}
    \phi(t)&=t^{qm'}\log(1+t)^{kqm'+(q-1)m'+\delta} \\
    \phi_i(t)&=t^{p_i(m_i'-1)}\log(1+t)^{(kp_i+1)(m_i'-1)+\delta}
\end{align*}
with arbitrarily small $\delta>0$. Straightforward computations show that the inverses of $\frac{t^\frac{1}{(mq)'}}{\phi^{-1}(t)\log(1+t)^k}$ and $\frac{t^\frac{m_i}{p_i}}{\phi_i^{-1}(t)\log(1+t)^k}$ are respectively the Young functions:
\begin{align*}
    \psi(t) & = t^{q'}\log(1+t)^{-1-\frac{\delta}{(q-1)m'}} \in B_{q'}, \\
    \psi_i(t) & = t^{p_i}\log(1+t)^{-1-(m_i-1)\delta} \in B_{p_i}\subset B_{p_i,\frac{p_iq}{p}}.
\end{align*}

Utilizing these notations, we have the following theorems.

\begin{theorem}\label{1pqComm}
Given $p_1,p_2>1$ and $1\leq p\leq q$, if
\begin{align}\label{Condition1pqComm}
    (u,v_1,v_2)\in \bigcap_{\vec k\in\mathsf{K}} \bigcup_{\substack{\overrightarrow m \, \in  \,\Lambda_{\vec p} \\ \vec \phi\in\mathscr{Y}^*_{\vec k,\overrightarrow m,\vec p}}}\A^{\vec\phi}_{\vec p}
\end{align}
then
\begin{align*}
    \|[\vec{b},\mathsf{BI_\al}]_{\vec{\beta}}\|_{L^q(u)} \lesssim \|\vec{b}\|_{\mathsf{BMO}}\,\|f\|_{L^{p_1}(v_1)} \|g\|_{L^{p_2}(v_2)}
\end{align*}
where $\|\vec{b}\|_{\mathsf{BMO}}\,=\prod_{i=1}^N\|b_i\|_{\mathsf{BMO}}$.
\end{theorem}

\begin{theorem}\label{pq1Comm}
Given $p_1,p_2>1$ and $p\leqslant q\leqslant1$, if
\begin{align*}
    (u,v_1,v_2)\in \bigcap_{\vec k\in\mathsf{K}} \bigcup_{\substack{\overrightarrow m \, \in  \,\Lambda_{\vec p} \\ \left(\phi_1\,,\,\phi_2\,,\,t^{qm'}\log(1+t)^{qm'(N-|\vec k|)}\right) \in\,\mathscr{Y}_{\vec k,\overrightarrow m,\vec p}}}\A^{\left(\phi_1\,,\,\phi_2\,,\,t^{qm'}\log(1+t)^{qm'(N-|\vec k|)}\right)}_{\vec p}
\end{align*}
where $|\vec k|\coloneqq k_1+k_2$, then
\begin{align*}
    \|[\vec{b},\mathsf{BI_\al}]_{\vec{\beta}}\|_{L^q(u)} \lesssim \|\vec{b}\|_{\mathsf{BMO}}\,\|f\|_{L^{p_1}(v_1)} \|g\|_{L^{p_2}(v_2)}.
\end{align*}
\end{theorem}

\begin{theorem}\label{p1qComm}
Given $p_1,p_2>1$ and $p<1<q\leqslant\frac{p}{1-p}$, if
\begin{align*}
    (u,v_1,v_2)\in  \bigcap_{\vec k\in\mathsf{K}} \bigcup_{\substack{\overrightarrow m \, \in  \,\Lambda_{\vec p} \\ \vec \phi\in\mathscr{Y}^*_{\vec k,\overrightarrow m,\vec p}}}\A^{\vec\phi}_{\vec p}
\end{align*}
then
\begin{align*}
    \|[\vec{b},\mathsf{BI_\al}]_{\vec{\beta}}\|_{L^q(u)} \lesssim \|\vec{b}\|_{\mathsf{BMO}}\,\|f\|_{L^{p_1}(v_1)} \|g\|_{L^{p_2}(v_2)}.
\end{align*}
\end{theorem}

The proof of Theorems \ref{1pqComm} -- \ref{p1qComm} are similar to the proof of Theorems \ref{1pq} -- \ref{p1q}. There would only be some difficulties at the beginning when we try to handle the effects of the commutators on $\mathsf{BI}_\al$. Below, we give proof for Theorem \ref{1pqComm} and leave the others for interested readers.

\begin{proof}[Proof of Theorem \ref{1pqComm}]
For simplicity, we may assume that $f$ and $g$ are non-negative $C_c^\infty$-functions. It was shown in section 5 of \cite{KC2017} that
\begin{align*}
\lefteqn{ \int_{\mathbb{R}^n}\bigl|[\vec{b}, \mathsf{BI}_\alpha]_{\vec{\beta}}(f,g)(x)\bigr| \thinspace h(x) \thinspace u(x)^\frac{1}{q} \thinspace dx} \hspace{5mm} \\
    & \lesssim \sum_{A\subseteq\{1,...,M\}}\sum_{B\subseteq\{M+1,...,N\}} \sum_{Q\in\mathscr{D}} |Q|^{\frac{\alpha}{n}-1} \\
    & \hspace{12mm} \int_Q\int_{|y|_\infty\leqslant\ell(Q)} \prod_{i\in \overline{A}}|b_i(x-y)-\lambda_i| \prod_{i\in \overline{B}}|b_i(x+y)-\lambda_i| f(x-y)g(x+y)\thinspace dy \\
    & \hspace{71mm} \prod_{i\in A\cup B}|b_i(x)-\lambda_i| \thinspace \thinspace h(x) \thinspace u(x)^\frac{1}{q} \thinspace dx.
\end{align*}
where $A\cup\overline{A}=\{1,...,M\}$, $B\cup\overline{B}=\{M+1,...,N\}$, and $\lambda_i=\lambda_i(Q)=\dashint_{3Q}b_i(x)dx$ for each $Q\in \mathscr{D}$ and each $i=1,...,N$. By utilizing the idea illustrated in Figure \ref{ChangeVar}, we have:
\begin{align*}
\lefteqn{ \int_{\mathbb{R}^n}\bigl|[\vec{b}, \mathsf{BI}_\alpha]_{\vec{\beta}}(f,g)(x)\bigr| \thinspace h(x) \thinspace u(x)^\frac{1}{q} \thinspace dx} \hspace{5mm} \\
    & \lesssim \sum_{A\subseteq\{1,...,M\}}\sum_{B\subseteq\{M+1,...,N\}} \sum_{Q\in\mathscr{D}} |Q|^{\frac{\alpha}{n}-1} \\
    & \hspace{27mm} \int_Q\left(f\prod_{i\in \overline{A}}|b_i-\lambda_i|\1_{3Q}\right)*\left(g\prod_{i\in \overline{B}}|b_i-\lambda_i|\1_{3Q}\right)(2x) \\
    & \hspace{68mm} \left(\prod_{i\in A\cup B}|b_i(x)-\lambda_i|\right) h(x) \thinspace u(x)^\frac{1}{q} dx \\
    & \leqslant \sum_{A\subseteq\{1,...,M\}}\sum_{B\subseteq\{M+1,...,N\}} \sum_{Q\in\mathscr{D}} |Q|^{\frac{\alpha}{n}-1} \\
    & \hspace{24mm} \left[\int_Q \left(f\prod_{i\in \overline{A}}|b_i-\lambda_i|\1_{3Q}\right)*\left(g\prod_{i\in \overline{B}}|b_i-\lambda_i|\1_{3Q}\right)(2x)^{mq} \, dx\right]^\frac{1}{mq} \\
    & \hspace{39mm} \left[\int_Q h(x)^{(mq)'} u(x)^{\frac{(mq)'}{q}} \prod_{i\in A\cup B}|b_i(x)-\lambda_i|^{(mq)'} dx\right]^\frac{1}{(mq)'} \\
    & \leqslant \sum_{A\subseteq\{1,...,M\}}\sum_{B\subseteq\{M+1,...,N\}} \sum_{Q\in\mathscr{D}} |Q|^{\frac{\alpha}{n}-1} \left(\int_{3Q} f^\frac{p_1}{m_1}\prod_{i\in \overline{A}}|b_i-\lambda_i|^\frac{p_1}{m_1}\right)^\frac{m_1}{p_1} \\
    & \hspace{8mm}\left(\int_{3Q} g^\frac{p_2}{m_2}\prod_{i\in \overline{B}}|b_i-\lambda_i|^\frac{p_2}{m_2}\right)^\frac{m_2}{p_2} \left(\int_Q h^{(mq)'} \, u^{\frac{(mq)'}{q}}\prod_{i\in A\cup B}|b_i-\lambda_i|^{(mq)'}\right)^\frac{1}{(mq)'} \\
    & \simeq \sum_{A\subseteq\{1,...,M\}}\sum_{B\subseteq\{M+1,...,N\}} \sum_{Q\in\mathscr{D}} |Q|^{\frac{\alpha}{n}+1} \, \left\|f\prod_{i\in \overline{A}}(b_i-\lambda_i)\right\|_{L^\frac{p_1}{m_1},3Q} \\
    & \hspace{40mm} \left\|g\prod_{i\in \overline{B}}(b_i-\lambda_i)\right\|_{L^\frac{p_2}{m_2},3Q}\, \left\|hu^\frac{1}{q}\prod_{i\in A\cup B}(b_i-\lambda_i)\right\|_{L^{(mq)'},Q} \\
    & \lesssim \sum_{A\subseteq\{1,...,M\}}\sum_{B\subseteq\{M+1,...,N\}} \sum_{t=1}^{2^n}\sum_{Q\in\mathscr{D}_t} |Q|^{\frac{\alpha}{n}+1} \, \left\|f\prod_{i\in \overline{A}}(b_i-\lambda_i)\right\|_{L^\frac{p_1}{m_1},Q} \\
    & \hspace{41mm} \left\|g\prod_{i\in \overline{B}}(b_i-\lambda_i)\right\|_{L^\frac{p_2}{m_2},Q}\, \left\|hu^\frac{1}{q}\prod_{i\in A\cup B}(b_i-\lambda_i)\right\|_{L^{(mq)'},Q}
\end{align*}
where we have utilized Theorem \ref{1/3trick} to obtain the last inequality. By the generalized H\"older inequality, we have the following estimates:

\begin{align*}
    \left\|f\prod_{i\in \overline{A}}(b_i-\lambda_i)\right\|_{L^\frac{p_1}{m_1},Q}
    & \lesssim \|f\|_{L^\frac{p_1}{m_1}(\log L)^{\frac{p_1}{m_1}|\overline{A}|},Q}\prod_{i\in \overline{A}} \|b_i-\lambda_i\|_{\exp(L),Q} \\
    & \lesssim \|f\|_{L^\frac{p_1}{m_1}(\log L)^{\frac{p_1}{m_1}|\overline{A}|},Q} \prod_{i\in \overline{A}} \|b_i\|_{\mathsf{\textsf{BMO}}} \\
    \left\|g\prod_{i\in \overline{B}}(b_i-\lambda_i)\right\|_{L^\frac{p_2}{m_2},Q} & \lesssim \|g\|_{L^\frac{p_2}{m_2}(\log L)^{\frac{p_2}{m_2}|\overline{B}|},Q} \prod_{i\in \overline{B}} \|b_i-\lambda_i\|_{\exp(L),Q} \\
    & \lesssim \|g\|_{L^\frac{p_2}{m_2}(\log L)^{\frac{p_2}{m_2}|\overline{B}|},Q}\prod_{i\in \overline{B}} \|b_i\|_{\mathsf{\textsf{BMO}}} \\
    \left\|hu^\frac{1}{q}\prod_{i\in A\cup B}(b_i-\lambda_i)\right\|_{L^{(mq)'},Q} & \lesssim \|hu^\frac{1}{q}\|_{L^{(mq)'}(\log L)^{{(mq)'}|A\cup B|},Q}\prod_{i\in A\cup B}\|b_i-\lambda_i\|_{\exp(L),Q} \\
    & \lesssim \|hu^\frac{1}{q}\|_{L^{(mq)'}(\log L)^{{(mq)'}|A\cup B|},Q} \prod_{i\in A\cup B} \|b_i\|_{\mathsf{\textsf{BMO}}}.
\end{align*}
These estimates allows us to continue our previous estimates as follows:
\begin{align*}
\lefteqn{ \int_{\mathbb{R}^n}\bigl|[\vec{b}, \mathsf{BI}_\alpha]_{\vec{\beta}}(f,g)(x)\bigr| \thinspace h(x) \thinspace u(x)^\frac{1}{q} \thinspace dx} \hspace{5mm} \\
    & \lesssim \|\vec{b}\|\sum_{A\subseteq\{1,...,M\}}\sum_{B\subseteq\{M+1,...,N\}} \sum_{t=1}^{2^n}\sum_{Q\in\mathscr{D}_t} |Q|^{\frac{\alpha}{n}+1} \, \|f\|_{L^\frac{p_1}{m_1}(\log L)^{\frac{p_1}{m_1}|\overline{A}|},Q} \\
    & \hspace{46mm} \|g\|_{L^\frac{p_2}{m_2}(\log L)^{\frac{p_2}{m_2}|\overline{B}|},Q} \, \|hu^\frac{1}{q}\|_{L^{(mq)'}(\log L)^{{(mq)'}|A\cup B|},Q}.
\end{align*}
Observe that the first three sums in the last display are finite sums, so we only need to gain control over the inner-most sum on a dyadic grid $\D_t$. Next, we will replace the sum on $\D_t$ by a sum on a sparse family of cubes. For every $k\in\mathbb{Z}$, let $\{Q_{k,j}\}_j$ be a collection of disjoint cubes from $\mathscr{D}_t$ that are maximal with respect to
$$
\|f\|_{L^\frac{p_1}{m_1}(\log L)^{\frac{p_1}{m_1}|\overline{A}|},Q} \thinspace \|g\|_{L^\frac{p_2}{m_2}(\log L)^{\frac{p_2}{m_2}|\overline{B}|},Q} > 4^{(n+2)k}.
$$
Define $E_{k,j}=Q_{k,j}\setminus\bigcup_iQ_{k+1,i}$. The family $\{E_{k,j}\}_{k,j}$ is pair-wise disjoint. Furthermore, let $P$ denote an immediate dyadic parent of $Q_{k,j}$, by the maximality of $Q_{k,j}$ and $Q_{k+1,i}$ we have
\begin{align*}
\lefteqn{\left|Q_{k,j} \cap \thinspace\bigcup_iQ_{k+1,i}\right| = \sum_{Q_{k+1,i}\subseteq Q_{k,j}}\bigl|Q_{k+1,i}\bigr|} \\
    & \leqslant \frac{1}{2^{(n+2)(k+1)}} \sum_{Q_{k+1,i}\subseteq Q_{k,j}} \bigl|Q_{k+1,i}\bigr| \, \|f\|^\frac{1}{2}_{L^\frac{p_1}{m_1}(\log L)^{\frac{p_1}{m_1}|\overline{A}|},Q_{k+1,i}} \|g\|^\frac{1}{2}_{L^\frac{p_2}{m_2}(\log L)^{\frac{p_2}{m_2}|\overline{B}|},Q_{k+1,i}} \\
    & \leqslant \frac{1}{2^{(n+2)(k+1)}} \left[\sum_{i} \bigl|Q_{k+1,i}\bigr| \thinspace \|f\|_{L^\frac{p_1}{m_1}(\log L)^{\frac{p_1}{m_1}|\overline{A}|},Q_{k+1,i}}\right]^\frac{1}{2} \\
    & \hspace{63mm} \left[\sum_{i} \bigl|Q_{k+1,i}\bigr| \|g\|_{L^\frac{p_2}{m_2}(\log L)^{\frac{p_2}{m_2}|\overline{B}|},Q_{k+1,i}}\right]^\frac{1}{2}.
\end{align*}
For any $\lambda,\mu>0$, we have
\begin{align*}
\lefteqn{\left|Q_{k,j} \cap \thinspace\bigcup_iQ_{k+1,i}\right|} \\
    & \leqslant \frac{1}{2^{(n+2)(k+1)}} \left[\sum_{i} \bigl|Q_{k+1,i}\bigr| \left(\lambda + \frac{\lambda}{\bigl|Q_{k+1,i}\bigr|} \int_{Q_{k+1,i}} \frac{|f|^\frac{p_1}{m_1}}{\lambda^\frac{p_1}{m_1}} \log\left(1+\frac{|f|}{\lambda}\right)^{\frac{p_1}{m_1}|\overline{A}|}\right)\right]^\frac{1}{2} \\
    & \hspace{26mm} \left[\sum_{i} \bigl|Q_{k+1,i}\bigr| \left(\mu + \frac{\mu}{\bigl|Q_{k+1,i}\bigr|}\int_{Q_{k+1,i}} \frac{|g|^\frac{p_2}{m_2}}{\mu^\frac{p_2}{m_2}} \log\left(1+\frac{|g|}{\mu}\right)^{\frac{p_2}{m_2}|\overline{B}|}\right)\right]^\frac{1}{2} \\
    & = \frac{1}{2^{(n+2)(k+1)}} \left[\sum_{i} \lambda \int_{Q_{k+1,i}}\left(1+\frac{|f|^\frac{p_1}{m_1}}{\lambda^\frac{p_1}{m_1}} \log\left(1+\frac{|f|}{\lambda}\right)^{\frac{p_1}{m_1}|\overline{A}|}\right)\right]^\frac{1}{2} \\
    & \hspace{50mm}\left[\sum_{i}\mu\int_{Q_{k+1,i}} \left(1+\frac{|g|^\frac{p_2}{m_2}}{\mu^\frac{p_2}{m_2}} \log\left(1+\frac{|g|}{\mu}\right)^{\frac{p_2}{m_2}|\overline{B}|} \right)\right]^\frac{1}{2} \\
    & \leqslant \frac{1}{2^{(n+2)(k+1)}} \left[\lambda \int_{Q_{k,j}}\left(1+\frac{|f|^\frac{p_1}{m_1}}{\lambda^\frac{p_1}{m_1}} \log\left(1+\frac{|f|}{\lambda}\right)^{\frac{p_1}{m_1}|\overline{A}|}\right)\right]^\frac{1}{2} \\
    & \hspace{58mm}\left[\mu\int_{Q_{k,j}} \left(1+\frac{|g|^\frac{p_2}{m_2}}{\mu^\frac{p_2}{m_2}} \log\left(1+\frac{|g|}{\mu}\right)^{\frac{p_2}{m_2}|\overline{B}|} \right)\right]^\frac{1}{2} \\
    & \leqslant \frac{2^n}{2^{(n+2)(k+1)}} \bigl|Q_{k,j}\bigr| \left[\lambda + \frac{\lambda}{\bigl|P\bigr|} \int_{P}\frac{|f|^\frac{p_1}{m_1}}{\lambda^\frac{p_1}{m_1}} \log\left(1+\frac{|f|}{\lambda}\right)^{\frac{p_1}{m_1}|\overline{A}|}\right]^\frac{1}{2} \\
    & \hspace{63mm} \left[\mu + \frac{\mu}{\bigl|P\bigr|} \int_{P}\frac{|g|^\frac{p_2}{m_2}}{\mu^\frac{p_2}{m_2}} \log\left(1+\frac{|g|}{\mu}\right)^{\frac{p_2}{m_2}|\overline{B}|}\right]^\frac{1}{2}
\end{align*}
where $P$ is the immediate dyadic parent of $Q_{k,j}$ in $\D_t$. By taking the infimum over all $\lambda,\mu>0$, we have
\begin{align*}
    \left|Q_{k,j} \cap \thinspace\bigcup_iQ_{k+1,i}\right| & \leqslant \frac{2^{n+1}}{2^{(n+2)(k+1)}} \bigl|Q_{k,j}\bigr| \, \|f\|^\frac{1}{2}_{L^\frac{p_1}{m_1}(\log L)^{\frac{p_1}{m_1}|\overline{A}|},P} \, \|g\|^\frac{1}{2}_{L^\frac{p_2}{m_2}(\log L)^{\frac{p_2}{m_2}|\overline{B}|},P}.
\end{align*}
By the maximality of $Q_{k,j}$, we have
\begin{align*}
    \left|Q_{k,j} \cap \thinspace\bigcup_iQ_{k+1,i}\right| \leqslant \frac{2^{n+1}}{2^{(n+2)(k+1)}} \bigl|Q_{k,j}\bigr| \, 2^{(n+2)k} = \frac12 \, \bigl|Q_{k,j}\bigr|
\end{align*}
which implies $\left|Q_{k,j}\right|\leqslant2\left|E_{k,j}\right|$, and hence the family $\mathscr{S}_t=\{Q_{k,j}:k\in\Z,j\in\Z\}$ is sparse. Let
$$C_k=\left\{Q\in\D_t: 4^{(n+2)k}<\|f\|_{L^\frac{p_1}{m_1}(\log L)^{\frac{p_1}{m_1}|\overline{A}|},Q} \|g\|_{L^\frac{p_2}{m_2}(\log L)^{\frac{p_2}{m_2}|\overline{B}|},Q} \leqslant 4^{(n+2)(k+1)}\right\}.$$

For any $\lambda>0$ we have the following estimates:
\begin{align*}
\lefteqn{\sum_{Q\in\mathscr{D}_t} |Q|^{\frac{\alpha}{n}+1} \, \|f\|_{L^\frac{p_1}{m_1}(\log L)^{\frac{p_1}{m_1}|\overline{A}|},Q} \, \|g\|_{L^\frac{p_2}{m_2}(\log L)^{\frac{p_2}{m_2}|\overline{B}|},Q} \, \|hu^\frac{1}{q}\|_{L^{(mq)'}(\log L)^{{(mq)'}|A\cup B|},Q}} \hspace{5mm}  \\
    & \leqslant \sum_{k\in\mathbb{Z}}\sum_{Q\in C_k} |Q|^{\frac{\alpha}{n}+1} \|f\|_{L^\frac{p_1}{m_1}(\log L)^{\frac{p_1}{m_1}|\overline{A}|},Q} \|g\|_{L^\frac{p_2}{m_2}(\log L)^{\frac{p_2}{m_2}|\overline{B}|},Q} \\
    & \hspace{79mm} \|hu^\frac{1}{q}\|_{L^{(mq)'}(\log L)^{{(mq)'}|A\cup B|},Q} \\
    & \leqslant \sum_{k\in\mathbb{Z}} 4^{(n+2)(k+1)} \sum_{Q\in C_k} |Q|^{\frac{\alpha}{n}+1} \thinspace \|hu^\frac{1}{q}\|_{L^{(mq)'}(\log L)^{{(mq)'}|A\cup B|},Q} \\
    & \leqslant \sum_{k\in\mathbb{Z}} 4^{(n+2)(k+1)} \sum_{j\in\mathbb{Z}} \sum_{\substack{Q\in \mathscr{D} \\ Q\subseteq Q_{k,j}}} |Q|^{\frac{\alpha}{n}+1} \thinspace \|hu^\frac{1}{q}\|_{L^{(mq)'}(\log L)^{{(mq)'}|A\cup B|},Q} \\
    & \leqslant \sum_{k\in\mathbb{Z}} 4^{(n+2)(k+1)} \sum_{j\in\mathbb{Z}} \sum_{\substack{Q\in \mathscr{D} \\ Q\subseteq Q_{k,j}}} |Q|^{\frac{\alpha}{n}+1} \\
    & \hspace{41mm} \left[ \lambda + \frac{\lambda}{|Q|}\int_{Q} \frac{\bigl|hu^\frac{1}{q}\bigr|^{(mq)'}}{\lambda^{(mq)'}}\log\left(1+\frac{\bigl|hu^\frac{1}{q}\bigr|}{\lambda}\right)^{{(mq)'}|A\cup B|}\right] \\
    & \leqslant \sum_{k\in\mathbb{Z}} 4^{(n+2)(k+1)} \sum_{j\in\mathbb{Z}} \sum_{r=0}^{\infty}\sum_{\substack{Q\in \mathscr{D}, \thinspace Q\subseteq Q_{k,j} \\ \ell(Q)=2^{-r}\ell(Q_{k,j})}} |Q|^{\frac{\alpha}{n}} \\
    & \hspace{45mm} \lambda\int_{Q} \left[1+\frac{\bigl|hu^\frac{1}{q}\bigr|^{(mq)'}}{\lambda^{(mq)'}}\log\left(1+\frac{\bigl|hu^\frac{1}{q}\bigr|}{\lambda}\right)^{{(mq)'}|A\cup B|}\right] \\
    & \leqslant \sum_{k\in\mathbb{Z}} 4^{(n+2)(k+1)} \sum_{j\in\mathbb{Z}} \lambda \left|Q_{k,j}\right|^\frac{\al}{n} \sum_{r=0}^{\infty} 2^{-r\al} \sum_{\substack{Q\in \mathscr{D}, \thinspace Q\subseteq Q_{k,j} \\ \ell(Q)=2^{-r}\ell(Q_{k,j})}} \\
    & \hspace{47mm} \int_{Q} \left[1+\frac{\bigl|hu^\frac{1}{q}\bigr|^{(mq)'}}{\lambda^{(mq)'}}\log\left(1+\frac{\bigl|hu^\frac{1}{q}\bigr|}{\lambda}\right)^{{(mq)'}|A\cup B|}\right] \\
    & \leqslant \frac{2^{\alpha}}{2^{\alpha}-1} \sum_{k\in\mathbb{Z}} 4^{(n+2)(k+1)} \sum_{j\in\mathbb{Z}} \left|Q_{k,j}\right|^{\frac{\al}{n}+1} \\
    & \hspace{32mm} \left[ \lambda + \frac{\lambda}{|Q_{k,j}|}\int_{Q_{k,j}} \frac{\bigl|hu^\frac{1}{q}\bigr|^{(mq)'}}{\lambda^{(mq)'}}\log\left(1+\frac{\bigl|hu^\frac{1}{q}\bigr|}{\lambda}\right)^{{(mq)'}|A\cup B|}\right].
\end{align*}
By taking the infimum over all $\lambda>0$, we have accomplished our goal transitioning from the sum on $\D_t$ to a sum over the sparse family of cubes $\mathscr{S}_t$. We have:
\begin{align*}
    & \sum_{Q\in\mathscr{D}_t} |Q|^{\frac{\alpha}{n}+1} \, \|f\|_{L^\frac{p_1}{m_1}(\log L)^{\frac{p_1}{m_1}|\overline{A}|},Q} \, \|g\|_{L^\frac{p_2}{m_2}(\log L)^{\frac{p_2}{m_2}|\overline{B}|},Q} \, \|hu^\frac{1}{q}\|_{L^{(mq)'}(\log L)^{{(mq)'}|A\cup B|},Q} \\
    & \leqslant \frac{2^{\alpha+1}}{2^{\alpha}-1} \sum_{k\in\mathbb{Z}} 4^{(n+2)(k+1)} \sum_{j\in\mathbb{Z}} \left|Q_{k,j}\right|^{\frac{\al}{n}+1} \|hu^\frac{1}{q}\|_{L^{(mq)'}(\log L)^{{(mq)'}|A\cup B|},Q_{k,j}} \\
    & \lesssim \sum_{k\in\mathbb{Z}} \sum_{j\in\mathbb{Z}} \left|Q_{k,j}\right|^{\frac{\al}{n}+1}  \, \|f\|_{L^\frac{p_1}{m_1}(\log L)^{\frac{p_1}{m_1}|\overline{A}|},Q_{k,j}} \, \|g\|_{L^\frac{p_2}{m_2}(\log L)^{\frac{p_2}{m_2}|\overline{B}|},Q_{k,j}} \\
    & \hspace{81mm} \|hu^\frac{1}{q}\|_{L^{(mq)'}(\log L)^{{(mq)'}|A\cup B|},Q_{k,j}} \\
    & = \sum_{Q\in\mathscr{S}_t} \left|Q\right|^{\frac{\al}{n}+1} \|f\|_{L^\frac{p_1}{m_1}(\log L)^{\frac{p_1}{m_1}|\overline{A}|},Q} \|g\|_{L^\frac{p_2}{m_2}(\log L)^{\frac{p_2}{m_2}|\overline{B}|},Q} \|hu^\frac{1}{q}\|_{L^{(mq)'}(\log L)^{{(mq)'}|A\cup B|},Q} \\
    & \lesssim \sum_{Q\in\mathscr{S}_t} |Q|^{\frac{\alpha}{n}+1} \, \|fv_1^\frac{1}{p_1}\|_{\psi_1,Q} \, \|v_1^{-\frac{1}{p_1}}\|_{\phi_1,Q} \, \|gv_2^\frac{1}{p_2}\|_{\psi_2,Q} \, \|v_2^{-\frac{1}{p_2}}\|_{\phi_2,Q} \, \|h\|_{\psi,Q} \, \|u^{\frac{1}{q}}\|_{\phi,Q} \\
    & \lesssim\sum_{Q\in\mathscr{S}_t} |Q|^{\frac{1}{p}+\frac{1}{q'}} \, \|fv_1^\frac{1}{p_1}\|_{\psi_1,Q} \, \|gv_2^\frac{1}{p_2}\|_{\psi_2,Q} \, \|h\|_{\psi,Q}
\end{align*}
where $\phi,\phi_1,\phi_2$ be the Young functions that comes from condition \eqref{Condition1pqComm}, and $\psi,\psi_1,\psi_2$ respectively denote the inverses of the functions $\frac{t^\frac{1}{(mq)'}}{\phi^{-1}(t)\log(1+t)^{|A\cup B|}}$, $\frac{t^\frac{m_1}{p_1}}{\phi_1^{-1}(t)\log(1+t)^{|\overline{A}|}}$ and $\frac{t^\frac{m_2}{p_2}}{\phi_2^{-1}(t)\log(1+t)^{|\overline{B}|}}$. By condition \eqref{Condition1pqComm}, we have $\psi\in B_{q'}$ and $\psi_i\in B_{p_i,\frac{p_iq}{p}}$ where we have applied $\vec k=\left(|\overline{A}|,|\overline{B}|\right)$. The rest of the proof would follow exactly as shown in the proof of Theorem \ref{1pq} and \ref{p1q}.
\end{proof}

\hfill

%-------------------------------------------------
% NEW SECTION: A Maximal Control Theorem
%-------------------------------------------------
\section{A Maximal Control Theorem} \label{Max}

For every pair of numbers $(r,s)\in[1,\infty)\times[1,\infty)$, we consider the maximal operators
$$\M_\al^{r,s}(f,g)(x)=\sup_{Q\ni x}|Q|^\frac{\al}{n} \left(\dashint_Q|f|^r\right)^\frac{1}{r} \left(\dashint_Q|g|^s\right)^\frac{1}{s}.$$
When $r=s=1$, this maximal operator reduces to the classical ``bilinear'' version of the Hardy-Littlewood maximal function
$$\M_\al(f,g)(x)=\M_\al^{1,1}(f,g)(x)=\sup_{Q\ni x}|Q|^\frac{\al}{n} \dashint_Q|f| \, \dashint_Q|g|.$$

The maximal operator $\M_\al^{r,s}$ was introduced and studied in \cite{KC2017}. The authors proved the following theorem for $p_1$ and $p_2$ be such that $p=\frac{p_1p_2}{p_1+p_2}>1$, but one can follow their proof and find that it actually works for all $p_1\geqslant r$ and $p_2\geqslant s$. Below we states the improved version of this theorem.

\begin{theorem} \label{2w-maximal-function-charaterization}
Suppose $0\leqslant\alpha<n$, $1<r\leqslant p_1$, $1<s\leqslant p_2$ and $p\leqslant q$. We have:
$$\M_\al^{r,s}:\,L^{p_1}(v_1)\times L^{p_2}(v_2)\to L^{q,\infty}(u)$$
if and only if
\begin{align*}
    \sup_Q |Q|^{\frac{\alpha}{n}+\frac{1}{q}-\frac{1}{p}}\left(\dashint_Q{u}\right)^{\frac{1}{q}} \left(\dashint_Q{v_1^{-\frac{r}{p_1-r}}}\right)^{\frac{p_1-r}{rp_1}} \left(\dashint_Q{v_2^{-\frac{s}{p_2-s}}}\right)^{\frac{p_2-s}{sp_2}} < \infty
\end{align*}
where $\left(\dashint_Q{v_1^{-\frac{r}{p_1-r}}}\right)^\frac{p_1-r}{r}=\left(\inf_Q{v_1}\right)^{-1}$ when $p_1=r$, and $\left(\dashint_Q{v_2^{-\frac{s}{p_2-s}}}\right)^\frac{p_2-s}{s}=\left(\inf_Q{v_2}\right)^{-1}$ when $p_2=s$.
\end{theorem}

Utilizing the new ideas in the preceding sections, we obtain the following theorem which improves Theorem 2.8 in \cite{KC2017} three-folded. First, the result is stated for all $\frac{rs}{r+s}<p\leqslant q$ (full range) instead of $1<p\leqslant q$. Second, the weighted condition needs no bump on the target weight $u$, more precisely: $\phi(t)=t^q$ which yields $\|u^\frac{1}{q}\|_{L^q,Q}$. Last, the bump on the component weights $v_1$ and $v_2$ are made more general by introducing better integrability conditions, namely $B_{p_i,\frac{p_iq}{p}}$, that properly contains $B_{p_i}$.

\begin{theorem} \label{Strong2w-4M}
Suppose $0\leqslant\al<n$, $1<r<p_1$, $1<s<p_2$ and $\frac{rs}{r+s}<p=\frac{p_1p_2}{p_1+p_2}\leqslant q$. For any set of weights
\begin{align}\label{MaxFunctionCondition}
    (u,v_1,v_2)\in \bigcup_{(\phi_1,\phi_2,t^q)\in\mathscr{Y}_{\left(\frac{p_1}{r},\frac{p_2}{s}\right),\vec p}}\A^{(\phi_1,\,\phi_2,\,t^q)}_{\vec p}
\end{align}
we have
\begin{align*}
\|\M_\al^{r,s}(f,g)\|_{L^q(u)}\lesssim\|f\|_{L^{p_1}(v_1)}\,\|g\|_{L^{p_2}(v_2)}.
\end{align*}
\end{theorem}

\begin{proof}[Proof of Theorem \ref{Strong2w-4M}]
Let
$$\M_\al^{r,s,\D}(f,g)(x)=\sup_{\substack{Q\in\D\\Q\ni x}}|Q|^\frac{\al}{n} \left(\dashint_Q|f|^r\right)^\frac{1}{r} \left(\dashint_Q|g|^s\right)^\frac{1}{s}$$
where $\D$ is a dyadic grid. We observe that
\begin{align*}
    \M_\al^{r,s,\D_t}(f,g)(x)\leqslant \M_\al^{r,s}(f,g)(x) \leqslant 6^{n-\al} \sum_{t=1}^{2^n}\M_\al^{r,s,\D_t}(f,g)(x)
\end{align*}
for all $f,g\in C_c^\infty(\R^n)$. Let $a>0$ to be chosen later. For each $k\in\Z$, let $\Omega_k=\{x\in\R^n:\,\M_\al^{r,s,\D}(f,g)(x)>a^k\}$, then $\Omega_k=\bigcup_jQ_{k,j}$ where $Q_{k,j}$ are pairwise-disjoint dyadic cubes that are maximal with respect to $|Q|^{\frac{\al}{n}}\left(\dashint_Q |f|^r\right)^{\frac1r}\left(\dashint_Q |g|^s\right)^{\frac1s}>a^k$. We claim that $\mathscr{S}=\{Q_{k,j}\}_{k,j}$ is sparse with an appropriate choice for $a$. The proof for this claim is similar to a part of the going-sparse process that we did in the proofs of Section \ref{ProofMain}. Therefore we have
\begin{align*}
\lefteqn{\int_{\R^n}\M_\al^{r,s,\D}(f,g)(x)^qu(x)\,dx} \hspace{12mm}  \\
& = \sum_{k\in\Z}\underset{\Omega_k\setminus\Omega_{k+1}}{\int}\M_\al^{r,s,\D}(f,g)(x)^qu(x)\,dx \\
& \leqslant \sum_{k\in\Z}a^{(k+1)q}\underset{\Omega_k\setminus\Omega_{k+1}}{\int}u(x)\,dx \\
& \lesssim \sum_{k\in\Z}\sum_{j}|Q_{k,j}|^\frac{q\al}{n}\left(\dashint_{Q_{k,j}}|f(y)|^rdy\right)^\frac{q}{r}\left(\dashint_{Q_{k,j}}|g(z)|^sdz\right)^\frac{q}{s}\int_{Q_{k,j}}u(x)\,dx \\
& \coloneq \sum_{Q\in\mathscr{S}}|Q|^\frac{q\al}{n}\left(\dashint_{Q}|f(y)|^rdy\right)^\frac{q}{r}\left(\dashint_{Q}|g(z)|^sdz\right)^\frac{q}{s}\int_{Q}u(x)\,dx
\end{align*}

Let $\phi_1,\phi_2$ be the Young functions that comes from condition \eqref{MaxFunctionCondition}, and let $\psi_1,\psi_2$ denote the inverses of $\frac{t^{1/r}}{\phi_1^{-1}(t)}$ and $\frac{t^{1/s}}{\phi_2^{-1}(t)}$ respectively, so $\psi_i\in B_{p_i,\frac{p_iq}{p}}$. By the proposed conditions of the theorem, we have
\begin{align*}
\lefteqn{\int_{\R^n}\M_\al^{r,s,\D}(f,g)(x)^qu(x)\,dx} \hspace{12mm} \\
& \lesssim \sum_{Q\in\mathscr{S}}|Q|^{\frac{q\al}{n}+1}\|fv_1^\frac{1}{p_1}\|^q_{\psi_1,Q}\,\|v_1^{-\frac{1}{p_1}}\|^q_{\phi_1,Q}\,\|gv_2^\frac{1}{p_2}\|^q_{\psi_2,Q}\,\|v_2^{-\frac{1}{p_2}}\|^q_{\phi_2,Q}\,\dashint_{Q}u(x)\,dx \\
& \lesssim \sum_{Q\in\mathscr{S}}|Q|^{\frac{q}{p}}\|fv_1^\frac{1}{p_1}\|^q_{\psi_1,Q}\,\|gv_2^\frac{1}{p_2}\|^q_{\psi_2,Q}.
\end{align*}
The rest of the proof just goes exactly the same as the last part in the proof of Theorem \ref{pq1}.
\end{proof}

We now state our maximal control theorem. We note that the latest known result (theorem 2.5 \cite{KC2017}) requires $\frac1r+\frac1s=1$, but our theorem below does not. However, our results is restricted to only $0<q\leqslant1$.

\begin{theorem} \label{MaxControl}
Let $0<q\leqslant1$ and $r,s,m\geqslant1$ be such that $\frac{1}{r}+\frac{1}{s}=1+\frac{1}{qm}$. If $w^{m'}\in \mathsf{A}_\infty$, then we have
$$\int_{\mathbb{R}^n} \bigl|BI_{\alpha}(f,g)(x)\bigr|^q \, w(x) \, dx \lesssim \int_{\R^n}\M_\al^{r,s}(f,g)(x)^q w(x)\,dx.$$
\end{theorem}

\begin{remark}
    When $q=1$, $m$ may equal $1$ which makes $m'=\infty$. In this case, the condition $w^{m'}\in \mathsf{A}_\infty$ would be understood as $w\in\mathsf{RH}_\infty$.
\end{remark}

When $m=\frac1q$, we have $r=s=1$, $\M_\al^{r,s}=\M_\al$ and $\mathsf{RH}_{m'}=\mathsf{RH}_{(1/q)'}$. In this case, Theorem \ref{MaxControl} coincides with Theorem 1.8 in \cite{Kabe2014}. 

\vspace{3mm}

When $m=\infty$, we have $m'=1$ and $\frac{1}{r}+\frac{1}{s}=1$. In this case, Theorem \ref{MaxControl} coincides with Theorem 2.5 in \cite{KC2017} for $q\leqslant1$.

\hfill

\begin{proof}[Proof of Theorem \ref{MaxControl}]
As indicated above, we only need to work with the case when $m\in\big(\frac1q,\infty\big)$. Since $q\leqslant1$ and $w^{m'}\in\mathsf{A}_\infty$, by Theorem \ref{RH} we know $w\in\mathsf{RH}_{m'}$. This means
$$
\left(\dashint_Qw^{m'}\right)^{\frac{1}{m'}} \leqslant C \thinspace \dashint_Qw \qquad \text{for all cubes} \,\, Q.
$$
Utilizing this fact and the Young's inequality, we have the following estimates:
\begin{align*}
    \int_{\mathbb{R}^n} \bigl|BI_{\alpha} & (f,g)(x)\bigr|^q \, w(x) \, dx \\
    & \leqslant \sum_{Q\in\mathscr{D}} |Q|^{q\left(\frac{\alpha}{n}-1\right)} \left(\int_Q \Big|\left[(f\1_{3Q})*(g\1_{3Q})\right](2x)\Big|^{qm} \, dx\right)^\frac{1}{m} \left(\int_Q w(x)^{m'} dx\right)^\frac{1}{m'} \\
    & \leqslant \sum_{Q\in\mathscr{D}} |Q|^{\frac{q\alpha}{n}+1} \left(\dashint_{3Q} |f(x)|^r \, dx\right)^\frac{q}{r} \left(\dashint_{3Q} |g(x)|^s \, dx\right)^\frac{q}{s} \left(\dashint_Q w(x)^{m'}dx\right)^\frac{1}{m'} \\
    & \lesssim \sum_{t=1}^{2^n} \sum_{Q\in\mathscr{D}_t} |Q|^{\frac{q\alpha}{n}+1} \left(\dashint_{Q} |f(x)|^r \, dx\right)^\frac{q}{r} \left(\dashint_{Q} |g(x)|^s \, dx\right)^\frac{q}{s} \left(\dashint_Q w(x)^{m'}dx\right)^\frac{1}{m'} \\
    & \lesssim \sum_{t=1}^{2^n} \sum_{Q\in\mathscr{D}_t} |Q|^{\frac{q\alpha}{n}} \left(\dashint_{Q} |f(x)|^r \, dx\right)^\frac{q}{r} \left(\dashint_{Q} |g(x)|^s \, dx\right)^\frac{q}{s} \int_Q w(x)\,dx \\
    & \lesssim \sum_{t=1}^{2^n} \sum_{Q\in\mathscr{S}_t} |Q|^{\frac{q\alpha}{n}} \left(\dashint_{Q} |f(x)|^r \, dx\right)^\frac{q}{r} \left(\dashint_{Q} |g(x)|^s \, dx\right)^\frac{q}{s} w(Q) \\
    & \lesssim \sum_{t=1}^{2^n} \sum_{Q\in\mathscr{S}_t} |Q|^{\frac{q\alpha}{n}} \left(\dashint_{Q} |f(x)|^r \, dx\right)^\frac{q}{r} \left(\dashint_{Q} |g(x)|^s \, dx\right)^\frac{q}{s} w(E_Q) \\
    & \lesssim \int_{\R^n}\M_\al^{r,s}(f,g)(x)^q \, w(x)\,dx
\end{align*}
where the third to last inequality is actually the going-sparse process that we performed in the proofs in Section \ref{ProofMain}, and the second to last inequality is due to a property of $\mathsf{A}_\infty$-weights as stated in Lemma \ref{AinfinityProperty}.
\end{proof}

\hfill

%-------------------------------------------------
% REFERENCES
%-------------------------------------------------
\bibliographystyle{plain}{\small\bibliography{bibtex}}

\hfill

\hfill

\end{document}